\nonstopmode
\documentclass[]{amsart}
\usepackage{amssymb}
\usepackage{amsxtra}
\usepackage{amscd}
\usepackage{graphicx}
\usepackage{epsfig}

\newtheorem*{proposition*}{Proposition}
\newtheorem{theorem}{Theorem}
\newtheorem*{theorem*}{Theorem}

\newtheorem*{lemma*}{Lemma}

\newtheorem*{corollary*}{Corollary}

\def\cit#1#2{\ifx#1!\cite{#2}\else#2\fi} 
\def\o{\circ\,}

\def\al{\alpha}
\def\be{\beta}
\def\ga{\gamma}
\def\de{\delta}
\def\ep{\varepsilon}
\def\ze{\zeta}
\def\et{\eta}

\def\ka{\kappa}

\def\rh{\rho}

\def\ph{\varphi}

\def\om{\omega}

\def\i{^{-1}}

\def\p{\partial}
\let\on=\operatorname

\def\L{\mathcal L}
\def\AMSonly#1{}

\def\p{\partial}
\let\on=\operatorname

\def\R{\mathbb R}

\def\<{\big\langle}
\def\>{\big\rangle \:}
\def\div{\operatorname{div}}

\begin{document}
\title{On Euler's equation and `EPDiff'}
\author{David Mumford and Peter W. Michor}
\address{
David Mumford:
Division of Applied Mathematics, Brown University,
Box F, Providence, RI 02912, USA}
\email{David\_Mumford@brown.edu}
\address{
Peter W. Michor:
Fakult\"at f\"ur Mathematik, Universit\"at Wien,
Oskar-Morgenstern-Platz 1, A-1090 Wien, Austria}
\email{Peter.Michor@univie.ac.at}
\begin{abstract}
We study a family of approximations to Euler's equation depending on two parameters 
$\ep,\et \ge 0$. When $\ep=\et=0$ we have Euler's equation and when both are positive we have 
instances of the class of integro-differential equations called EPDiff in imaging science. These 
are all geodesic equations on either the full diffeomorphism group 
$\on{Diff}_{H^\infty}(\mathbb R^n)$ or, if $\ep = 0$, its volume preserving subgroup. They are 
defined by the right invariant metric induced by the norm on vector fields given by      
$$ \|v\|_{\ep,\et} = \int_{\mathbb R^n} \langle L_{\ep,\et} v, v \rangle\, dx $$ 
where $L_{\ep,\et} = (I-\tfrac{\et^2}{p} \triangle)^p \circ (I-\tfrac1{\ep^2} \nabla \circ \div)$. 
All geodesic equations are locally well-posed, and the $L_{\ep,\et}$-equation admits solutions for 
all time if $\et>0$ and $p\ge (n+3)/2$. We tie together solutions of all these equations by  
estimates which, however, are only local in time. This approach leads to a new notion of momentum 
which is transported by the flow and serves as a generalization of vorticity. We also discuss how 
delta distribution momenta lead to``vortex-solitons", also called ``landmarks" in imaging science, 
and to new numeric approximations to fluids.     
\end{abstract}
\keywords{diffeomorphism groups, Euler's equation, EPDiff, vortex-solitons, vortons}
\subjclass[2010]{Primary 35Q31, 58B20, 58D05} 

\maketitle

In Arnold's famous 1966 paper \cite{Arnold66}, he showed that Euler's equation in $\R^n$ for 
incompressible, non-viscous flow was identical to the geodesic equation on the group of volume 
preserving diffeomorphisms for the right invariant $L^2$-metric. This raises the question, what are 
the equations for geodesic flow on the full group of diffeomorphisms in various right invariant 
metrics? Arnold also gave the general recipe for writing down these equations but, as far as we 
know, geodesics of this sort were not specifically studied beyond the 1-dimensional case, until 
Miller and Grenander and co-workers introduced them into medical imaging applications. In 1993 they 
laid out a program for comparing individual medical scans with standard human body templates 
\cite{MillerEtAl93}. Subsequently they introduced a large class of right-invariant metrics on the 
group of (suitably smooth) diffeomorphisms using norms on vector fields given by: 
$$ \| v\|_L^2 = \int_{\R^n} \langle L v,  v \rangle dx.$$
Here $L$ is a positive definite self-adjoint differential operator. They proposed to measure the 
distance from the subject scan to the template by the length of the $L$-geodesic connecting them 
(see their survey article \cite{MillerEtAl02}). The geodesic equation for these metrics are 
integro-differential equations called EPDiff (or `Euler-Arnold' equations). In this paper we want 
to study the relationship of Euler's equation to EPDiff.    

To be specific, we shall use in this paper the group $\on{Diff}_{H^\infty}(\mathbb R^n)$ of all 
diffeomorphisms $\ph$ of the form $\ph(x) = x + f(x)$ with $f$ in the intersection $H^\infty$ of 
all Sobolev spaces $H^s$, $s\ge0$, and also its normal subgroup 
$\on{Diff}_{\mathcal S}(\mathbb R^n)$ where $f$ is in the space $\mathcal S$ of all rapidly decreasing 
functions. See \cite{MM12} for Lie group structures on them.    

Note that, in the $H^\infty$ case, $f$, along with its derivatives, will approach 0 as 
$\|x\|\rightarrow \infty$, but not necessarily at any fixed rate. The geodesic equation in these 
metrics is similar in form to fluid flow equations except that it involves a {\it momentum} $ m( x, 
t) = L v( x,t)$, called momentum because it is dual to velocity in the sense that the `energy' can 
be expressed as $\| v\|^2_L = \int \langle  v,  m\rangle dx$.

The geodesic equation EPDiff of interest in this paper is this:
\begin{equation}\label{eq:EPDiff}
\boxed{\p_t  m = -( v \cdot \nabla) m - \div( v) m -  m \cdot (D v)^t}
\end{equation}
In coordinates, we can write the right hand side more explicitly as:
$ -\sum_j (v_j \p_{x_j}m_i + \p_{x_j}v_j\cdot m_i + m_j \p_{x_i} v_j)$. Note that $ v$ can be 
recovered from $ m$ as $ v = K \ast  m$ where $K$ is the (matrix-valued) Green's function for the 
operator $L$, that is, its inverse in the space $\mathcal S$.   

The rather complicated expression for the rate of change of momentum -- that is the force -- has a 
simple meaning. Namely, let the vector field $ v$ integrate to a flow $\ph$ via 
$$ \p_t \ph(x, t) = v(\ph(x,t),t)$$
and describe the momentum by a {\it measure-valued 1-form} 
$$
\widetilde{m} = \sum_i m_i dx_i\otimes (dx_1\wedge \cdots \wedge dx_n)
$$
so that $\|v\|_L^2 = \int ( v, \widetilde m )$ makes intrinsic sense. 
Then it's not hard to check that equation \thetag{\ref{eq:EPDiff}} is equivalent to: 
$\widetilde m$ is {\it invariant} under the flow $\ph$, that is,
$$\widetilde{m}(\cdot,t) = \ph(\cdot, t)_* \widetilde{m}(\cdot,0),$$
whose infinitesimal version is the following, using the Lie derivative (see \cite[3.4]{M127}),
\begin{equation}
\p_t \widetilde m(\cdot,t) = - \mathcal L_{ v(\cdot,t)}\widetilde m(\cdot,t).\label{eq:mtilde}
\end{equation}
Because of this invariance, if a geodesic begins with momentum of compact support, it will always 
have compact support; and if it begins with momentum which, along with all its derivatives, has 
`rapid' decay at infinity, that is it is in $O(\|x\|^{-n})$ for every $n$, this too will persist. 
This comes from the lemma:   

\begin{lemma*} {\rm \cite{MM12}}  If $\ph \in \on{Diff}_{H^\infty}(\R^n)$ and $T$ is any smooth 
tensor on $\R^n$ with rapid decay at infinity, then $\ph_*(T)$ is again smooth with rapid decay at 
infinity.  
\end{lemma*}
Moreover this invariance gives us a Lagrangian form of EPDiff:
\begin{multline} \p_t \ph(x,t) = \int K^{\ph(\cdot,t)}(x,y) \large(\ph(y,t)_*\widetilde{m}(y,0)\large)
= K^{\ph(\cdot,t)} \ast (\ph(\cdot,t)_*\widetilde m(\cdot,0)) \\ 
\text{ where } K^\ph(x,y)=K(\ph(x),\ph(y)) \label{eq:lagrange}
\end{multline}

The main result of this paper is that solutions of Euler's equation are limits of solutions of 
equations in the EPDiff class with the operator: 
\begin{equation}
L_{\ep,\et} = (I-\tfrac{\et^2}{p}\triangle)^p \circ (I-\tfrac{1}{\ep^2} \nabla\circ \div), 
\qquad \text{for any } \ep > 0, \et \ge 0.\label{eq:epet}
\end{equation} 
We will show that all solutions of Euler's equation are limits of solutions of these much more 
regular EPDiff equations and {\it give a bound on their rate of convergence}. In fact, so long as 
$p>n/2+1$, Trouv\'e and Younes have shown \cite{TrouveYounes} that these EPDiff equations have a 
well-posed initial value problem with unique solutions for all time. Combining our result with 
theirs gives a new way of approximating solutions of Euler's equation by solutions of a more 
regular equation. Moreover, although $L_{0,\et}$ does not make sense, the analog of its Green's 
function $K_{0,\et}$ does make sense as do the equations (\ref{eq:EPDiff}), (\ref{eq:mtilde}). 
These are, in fact, geodesic equations on the group of volume preserving diffeomorphisms SDiff and 
become Euler's equation for $\et=0$. An important point is that so long as $\et >0$, the equations 
have `particle' solutions in which the momentum is a sum of delta functions.         

Our approach is closely related to several strands of work reformulating Euler's equation in a 
Hamiltonian setting. The first goes back to P.\ H.\ Roberts' 1972 paper \cite{Roberts} on weakly 
interacting vortex rings in $\R^3$  where a finite dimensional Hamiltonian system for a finite set 
of such rings is introduced (his equations (35) and (36)). In 1988 Oseledets\cite{Oseledets} gave a 
completely general Hamiltonian reformulation of Euler's equation. He introduced the dual momentum 
variables $ m(x,t)$ described above, called $\ga(x,t)$ in his paper, that have non-zero divergence 
in general. With a suitable Hamiltonian, he recovers Euler's equation as a Hamiltonian system. He 
notes that when the momenta are sums of delta functions, one recovers Roberts' system. 
Subsequently, in a second direction, Alexandre Chorin and his students Thomas Buttke, Ricardo 
Cortez and Michael Minion developed the discrete approximation by vortex rings as an effective way 
to solve Euler's equation numerically in line with the general ``Smoothed Particle Hydrodynamics" 
(SPH) technique (see \cite{Buttke,Cortez}). Chorin discusses this technique in \S 1.4 of his book 
\cite{Chorin} where he calls the momentum variables `magnetization'. Finally there is a third 
strand connected to our work. A key point in EPDiff is the use of operators $L$ of the form 
$(I-\triangle)^p$ which have the effect of smoothing the velocities $ v$ that solve the equation. 
The case $p=1$ arose earlier from the study of the Camassa-Holm equation \cite{CamHolm}, also 
called the $\al$-Euler equation for incompressible flows in dimension bigger than one. The CH 
equation is very explicitly related to EPDiff in Holm and Marsden's 2003 paper \cite{HolmMars}, 
which strongly motivated the present paper.                    

The main point here is that all this work, in both the discrete and continuous cases, fits in 
logically as special cases of the general EPDiff setup and thus as geodesic equations on the group 
of diffeomorphisms with Riemannian metrics depending on two auxiliary parameters. Besides these 
formal connections we give what we believe are new existence theorems for certain cases of EPDiff 
that, as stated above, lead to explicit bounds on the convergence of the particle methods to 
solutions of Euler's equation. In the last section, we show that Roberts' dynamics of vortex rings 
is the same as our geodesic dynamics when the momentum is a sum of delta distributions. In this 
context, it is interesting that this dynamical system generates in many cases higher order 
singularities in the infinite time limit and we illustrate these.        

The authors wish to thank Alexandre Chorin, Constantine Dafermos, Darryl Holm and Andr\'e Nachbin 
for very helpful conversations and references.

\section{Oseledets's form of Euler's equation} 
Oseledets' Hamiltonian formulation of Euler's equation states that for a suitable kernel $K$, 
Euler's equation becomes equation (\ref{eq:mtilde}) above. To describe his result, consider how 
EPDiff might be related to Euler. First of all, it's natural to take $K$ to be the identity matrix 
times a delta function $\de_0(x)$ because then $\| v\|^2_L$ is just the kinetic energy 
$\int | v(x)|^2dx$ where $|\cdot|$ denotes also the Euclidean norm in $\mathbb R^n$. 
Then $ v = m$, and EPDiff becomes:
$$ \p_t  v = - ( v \cdot \nabla)  v - \div( v)  v - \nabla (\tfrac12 | v|^2)$$
which looks like Euler's equation if the divergence of $ v$ can be made to be zero for all time and 
the last term can be interpreted as the gradient of pressure. But how do we keep the divergence of 
$ v$ zero? In fact, the right link between Euler and EPDiff is a little more subtle and requires 
the ansatz:   
$$ K_{ij}(x)= \de_{ij} \de_0(x) + \p_{x_i} \p_{x_j} H$$
with the Hessian of an auxiliary function $H$. With this form of $K$, we get:
$$ v_i-m_i = \sum_j (\p_{x_i}\p_{x_j} H)\ast m_j = \p_{x_i}\left( H \ast \sum_j \p_{x_j} m_j \right)$$
or
$$  v =  m + \nabla(H \ast \div( m)).$$
Substituting this into EPDiff and assuming $\div( v)=0$ we again get Euler's equation:
\begin{align*}
0 &= \p_t  m + ( v \cdot \nabla) m + \div( v) m +  m \cdot (D v)^t \\
&= \p_t  v - \nabla(\p_t (H \ast \div( m))) + ( v \cdot \nabla) v - ( v \cdot \nabla)(\nabla (H \ast \div( m)) +\\
&\qquad + \nabla(\tfrac12 | v|^2) - \nabla(H\ast \div( m))\cdot (D v)^t\\
&= \p_t  v + ( v \cdot \nabla) v + \nabla p\\
\text{with the pressure}&\\
p &= -\p_t (H \ast \div( m)) + \tfrac12 | v|^2 - ( v \cdot \nabla)(H \ast \div( m)).
\end{align*}
But now we can also guarantee that the divergence of $ v$ is zero if we choose $H$ correctly. We have:
$$\div( v) = \div( m) + \triangle(H\ast\div( m))= \div\left((\de_0 + \triangle H)\ast  m\right) $$
so all we need to do is to take $H$ to be the Green's function of minus the Laplacian and, at least 
formally, we get Euler's equation. But $K$ now has a rather substantial pole at the origin. In 
fact, if $V_n$ is the area of an $(n-1)$-sphere, then:  
\begin{align*}
H( x) &= \left\{\begin{array}{ll} 
\frac{1}{(n-2)V_n} (1/| x|^{n-2}) & \text{if } n>2, \\
\frac{1}{V_2} \log(1/| x|) & \text{if } n=2 
\end{array}\right.\\
\text{so that, {\it as a function}}& \\
\p_{x_i}\p_{x_j}H(x) &= \frac{1}{V_n} \cdot \dfrac{nx_i x_j - \de_{ij} | x|^2}{| x|^{n+2}}, \quad \text{if }x \ne 0.
\end{align*}
Letting $M_0$ denote this matrix-valued function, note that convolution with any of its elements 
$(M_0)_{ij}$ is still a  
Calderon-Zygmund singular integral operator defined by the limit as $\ep\rightarrow 0$ of its value 
outside an $\ep$-ball, so it is reasonably well behaved.  
{\it As a distribution} there is another term; compare with 
\cite[Thms. 3.2.4 and 3.3.2]{Hoermander83I}.
It is not hard to check that:
$$ 
\p_{x_i}\p_{x_j}H \stackrel{\text{distribution}}{=} (M_0)_{ij} - \frac{1}{n} \de_{ij} \de_0 
$$
Now convolution among distributions is associative and commutative so we have
$$ m + \p^2(H)_{\text{distr}} \ast m = m + \nabla H \ast \div(m)$$
which is the identity if $\div(m)=0$ and has values with $\div=0$, i.e., 
is a projection onto the subspace of divergence-free vector fields. 
As it is self-adjoint we see that
$$ m\mapsto  v =  \left(m + \p^2(H)_{\text{distr}}\right) = \left(\tfrac{n-1}{n}\cdot m + M_0 \ast  m\right)$$ 
is the \emph{orthogonal projection} $P_{\div=0}$ of the space of vector fields $ m$ 
onto the subspace of divergence free vector fields $v$. 
This is the vector form of the Hodge decomposition of 1-forms and is bounded orthogonal in each Sobolev space. 
The matrix given by the value of $M_0$ at each point $x \in \R^n$ has $\R x$ as an eigenspace with 
eigenvalue $(n-1)/V_n|x|^n$ and $\R x^\perp$ as an eigenspace with eigenvalue $-1/V_n|x|^n$. 
So if we let $P_{\mathbb Rx}$ and $P_{\mathbb Rx^\bot}$ be the orthonormal projections onto the 
eigenspaces,
then we have the useful formula:
\begin{multline*} 
P_{\div=0}(m)(x) =
\\
 \tfrac{n-1}{n}\cdot m(x) + \frac{1}{V_n}\cdot \lim_{\ep \rightarrow 0}\int_{|y|\ge \ep} \frac{1}{|y|^n} 
 \left( (n-1) P_{\mathbb Ry}(m(x-y)) - P_{\mathbb Ry^\perp}(m(x-y)\right)dy.
 \end{multline*}

With this choice of $K$, EPDiff in the variables $( v,  m)$ becomes the Euler equation in $ v$ 
with pressure given by an explicit function of $ m$ and $ v$. 
This gives us {\it Oseledets's form for Euler's equation}:
\begin{equation}\boxed{\begin{aligned}
v &= P_{\div=0}( m)\\
\p_t  m &= -( v \cdot \nabla) m -  m \cdot (D v)^t 
\end{aligned}}\label{eq:E1} \end{equation}
Let $\widetilde{m} = \sum_i m_i dx_i$ be the 1-form associated to the vector field $m$.
Since $v$ is divergence free we can use $\widetilde m$ instead of 
$\sum_i m_i dx_i \otimes dx_1\wedge\dots dx_n$. In integrated form, we have:
\begin{equation}\boxed{\begin{aligned}
\p_t \ph &=  P_{\div=0}( m) \circ \ph \\
\widetilde{m}(\cdot,t) &= \ph(\cdot, t)_* \widetilde{m}(\cdot,0)
\end{aligned}}\label{eq:E2}\end{equation}
This form of Euler's equation uses the variables $ v, m$ instead of the traditional $ v, p$ 
(velocity and pressure) but it has the great advantage that $ m$, like vorticity, is constant when 
suitably transported by the flow. In fact, $ m$ determines the vorticity, defined in arbitrary 
dimensions as the 2-form $\om = d(\sum_i v_i dx_i)$. This is because $ v$ and $ m$ differ by a 
gradient, so $\om = d\widetilde{m}$ also. Thus the fact that vorticity is constant along flows is a 
consequence of the same fact for the momentum 1-form $\widetilde{m}$. This way of writing the 
velocity field as a convolution with a momentum field means we write the velocity field as a 
superposition of the elementary vector fields $P_{\div = 0}( m_0 \de_{x_0})$ for all points $x_0$ 
and vectors $ m_0$. In dimension 2, $x_0=(0,0),  m_0 = (1,0)$, this is the harmonic vector field $ 
v = \left(\tfrac{x^2-y^2}{| x|^4}, \tfrac{2xy}{| x|^4}\right)$  with a singularity at 0 where it 
has a dipole as vorticity. In dimension 3, this vector field is an infinitesimal vortex ring which 
is how Roberts' paper \cite{Roberts} connects to our paper. 

One of the motivations for this formulation of Euler's equation is that if $v(x,0)$ is any initial 
condition for velocity, we take any momentum $m(x,0)$ such that $v(\cdot,0) = P_{\div = 
0}(m(\cdot,0))$. As Chorin has pointed out, in many situations one can start with $m(\cdot,0)$ of 
compact support and then $m$ will remain of compact support even though $v$ will have heavy tails 
due to the effects of pressure far from the support of $m$. This seems to be one of the reasons why 
his numerical vortex dipole/vortex ring technique works so well.     

\section{Approximating Euler with EPDiff}

However, the above equations (\ref{eq:E1}) and (\ref{eq:E2}) are not part of the true EPDiff framework 
because the operator $K = P_{\div = 0}$ is not invertible and there is no corresponding 
differential operator $L$. 
In fact, $ v$ does not determine $ m$ as we have rewritten Euler's equation using extra non-unique 
variables $m$,  albeit ones which obey a conservation law so they 
may be viewed simply as extra parameters. 
The simplest way to perturb this $K$ to make it invertible is to replace the above Green's function $H$ 
of the Laplacian by the Green's function 
$H_\ep$ of the positive definite operator $\ep^2 I - \triangle$ 
for some constant $\ep>0$ (whose dimension is $\text{length}^{-1}$). 
The Green's function may be given explicitly using the `K' Bessel function via the formula 
$$H_\ep( x) = c_n \ep^{n-2} |\ep x|^{1-n/2}K_{n/2-1}(|\ep x|)$$ 
for a suitable constant $c_n$ independent of $\ep$ (see \cite{AbramowitzStegun}).
 
Then we get the modified kernel 
$$(K_\ep)_{ij} = \de_{ij}\de_0 + (\p_{x_i}\p_{x_j}H_\ep)_{\text{distr}}$$  
This has exactly the same highest order pole at the origin as $K$ did and the second derivative is 
again a Calderon-Zygmund singular integral operator minus the same delta function. The main 
difference is that this kernel has exponential decay at infinity, not polynomial decay. By 
weakening the requirement that the velocity be divergence free, the resulting integro-differential 
equation behaves much more locally, more like a hyperbolic equation rather than a parabolic one.     

Note that here $K_\ep$ scales as $K_\ep( x) = \ep^n K_1(\ep x)$ and that, as $\ep$ goes to zero, 
the limit of $K_\ep$ (as an operator on $\mathcal S$, say) is just our previous kernel $K$.  
Taking the Fourier transform and inverting, we can find the corresponding operator $L_\ep$ in 
several steps: 
\begin{align*}
 \widehat{H_\ep} &= \frac{1}{\ep^2+|\xi|^2}, \quad \text{hence} \\
 \widehat{\p_{x_i}\p_{x_j} H_\ep} &= -\frac{\xi_i \xi_j}{\ep^2+|\xi|^2}, \quad \text{hence}\\
 \widehat{(K_\ep)_{ij}} &= \frac{(\ep^2 +|\xi|^2)\de_{ij}-\xi_i\xi_j}{\ep^2 + |\xi|^2} 
\end{align*}
Now the inverse of this {\it as a matrix} is the remarkably simple:
$$\left( \frac{(\ep^2 +|\xi|^2)\de_{ij}-\xi_i\xi_j}{\ep^2 + |\xi|^2} \right)^{-1} = \de_{ij} + \frac{\xi_i \xi_j}{\ep^2}$$
and this comes from the differential operator:
$$L_\ep = I - \tfrac{1}{\ep^2} \nabla \o \div$$
Thus we have inverse operators as required by the EPDiff setup:
$$  v = K_\ep \ast  m, \quad  m = L_\ep( v).$$ 
Finally this operator $L_\ep$ defines the simple metric:
$$ \| v\|^2_{L_\ep} = \int_{\R^n}\!\! \<  v, L_\ep( v)\> dx_1\wedge \cdots \wedge dx_n 
= \int_{\R^n}\!\! \left(|v( x)|^2 + \tfrac{1}{\ep^2} \div(v)( x)^2 \right)dx_1\wedge \cdots \wedge dx_n.$$

As in Arnold's original paper, formally at least, solutions of EPDiff for this $K_\ep, L_\ep$ are 
geodesics in the group of diffeomorphisms for this metric. EPDiff is the geodesic equation with 
momentum and velocity but in this case it simplifies to a form involving only velocity that closely 
resembles Euler's equation. Substituting the formula for $L_\ep$, we calculate as follows:   
\begin{align*}
\p_t(v_i) &= (K_\ep)_{ij} \ast \p_t(m_j) \\
&= -(K_\ep)_{ij} \ast \left( (v_k m_{j,k} + m_j\cdot \div( v) + m_k \cdot v_{k,j} \right) \\
&= -(K_\ep)_{ij} \ast \big( (v_k v_{j,k} + v_j \div( v) + v_k \cdot v_{k,j}) \\
& \qquad\qquad - \tfrac{1}{\ep^2}(v_k \div( v)_{,jk} + \div( v)_{,j} \div( v) + \div( v)_{,k} v_{k,j} )\big)\\
&= -(K_\ep)_{ij}\! \ast \!\Big(v_k v_{j,k} + \tfrac12\left(| v(x)|^2 
- (\tfrac{\div( v)}{\ep})^2\right)_{,j} + v_j \div( v) -\tfrac{1}{\ep^2} (v_k \div( v)_{,k})_{,j} \Big )\\
&= -(K_\ep)_{ij} \ast \Big(v_k v_{j,k} + \tfrac12\left(| v(x)|^2 + (\tfrac{\div( v)}{\ep})^2\right)_{,j} 
+ (L_\ep)_{jk}(v_k \div( v))  \Big )\\
&= -(K_\ep)_{ij} \ast(v_k v_{j,k}) -v_i \div( v)-\tfrac12 (K_\ep)_{ij}\ast \left(| v(x)|^2+(\tfrac{\div(v)}{\ep})^2\right)_{,j}
\end{align*}
Here we have written $| v(x)|$ in order to make clear that we are taking the norm 
of the single vector $ v(x)$, not the norm of the whole vector field, so $| v(x)|$ is a function on $\R^n$. 
Now we also have the identity:
$$ (K_\ep \ast \nabla f)_i = f_{,i} + \sum_j \p_i \p_j H_\ep \ast f_{,j} 
= f_{,i} + \triangle H_\ep \ast f_{,i}= \ep^2 \p_{x_i} H_\ep \ast f $$
so the final geodesic equation is:
\begin{equation}\boxed{\p_t( v)=-(K_\ep) \ast (( v \cdot \nabla)  v) 
- v\cdot \div( v)-\tfrac{\ep^2}{2} \nabla  H_\ep \ast \left(| v(x)|^2+(\tfrac{\div(v)}{\ep})^2 \right)}
\label{eqn:epsEuler}\end{equation}

This is certainly the simplest choice for a metric which allows non-zero divergence but, as 
$\ep \rightarrow 0$, seeks to make the divergence smaller and smaller so that, in the limit, the  
divergence must be identically zero and we have the $L^2$ metric on the group of volume preserving 
diffeomorphisms. At the same time, the above equation approaches Euler's equation. We will show 
below that solutions of the above equation must approach solutions of Euler's equation and, when 
the momentum has rapid decay at infinity, we will give an explicit bound on the rate of 
convergence. Curiously though, the parameter $\ep$ can be scaled away. That is, if 
$ v( x, t),  m( x, t)$ is a solution of EPDiff for the kernel $K_1$, then $ v(\ep  x, \ep t),  
m(\ep  x, \ep t)$ is  a solution of EPDiff for $K_\ep$. 

The above case of EPDiff still has a singular kernel $K_\ep$ for which existence theorems are 
difficult (see below). The cases of EPDiff which have been analyzed and used in medical imaging 
applications \cite{MillerEtAl02, TrouveYounes, Younes10} involve kernels which are $C^1$. We can 
easily make our singular example a limit of better behaved examples. The simplest way is to compose 
the above operator $L_\ep$ with  a scaled version of the standard regularizing kernel 
$(I-\triangle)^p$ giving the positive definite self-adjoint differential operator given above 
(equation (\ref{eq:epet}) of the Introduction):      
$$ L_{\ep,\et} = (I-\tfrac{\et^2}{p} \triangle)^p \circ (I-\tfrac1{\ep^2} \nabla \circ \div)$$
Here the constant $\et$ has dimension $\text{length}$ and although $\ep$ and $\et$ could be scaled 
away by themselves, the composite kernel has a dimensionless parameter $\et\cdot\ep$. Since 
$L_{\ep,\et}$ is a composition, so is its inverse and hence the kernel is now the convolution: 
$$ K_{\ep,\et} = G_\et^{(p)} \ast K_\ep$$
where $G_\et^{(p)}$ is the Green's function of $(I-\tfrac{\et^2}{p} \triangle)^p$ and is again 
given explicitly by a `K'-Bessel function $d_{p,n}\et^{-n}| x|^{p-n/2}K_{p-n/2}(| x|/\et)$. The 
reason for inserting $p$ in the denominator of the coefficient is that for $p\gg 0$, the kernel 
converges to a Gaussian with variance depending only on $\et$, namely 
$(2\sqrt{\pi}\et)^{-n}e^{-| x|^2/4\et^2}$.  
This follows because the Fourier transform takes $G_\et^{(p)}$ to 
$\left(1+\tfrac{\et^2|\xi|^2}{p}\right)^{-p}$, whose limit, as $p \rightarrow \infty$, is 
$e^{-\et^2|\xi|^2}$.  These approximately Gaussian kernels lie in $C^q$ if $q \le p-(n+1)/2$. So 
long as the kernel is in $C^1$, it is known that EPDiff has solutions for all time 
\cite{TrouveYounes}. A particularly simple case is when $p=(n+3)/2$. Then the Green's function is 
just a constant times the $C^2$ function $(1+|x|/\et)e^{-|x|/\et}$ as you can verify by taking 
$n=1$ and checking that that this is the Green's function of $1-\et^2 (d/dx)^2$.          
\begin{figure}[h!]
\epsfig{file=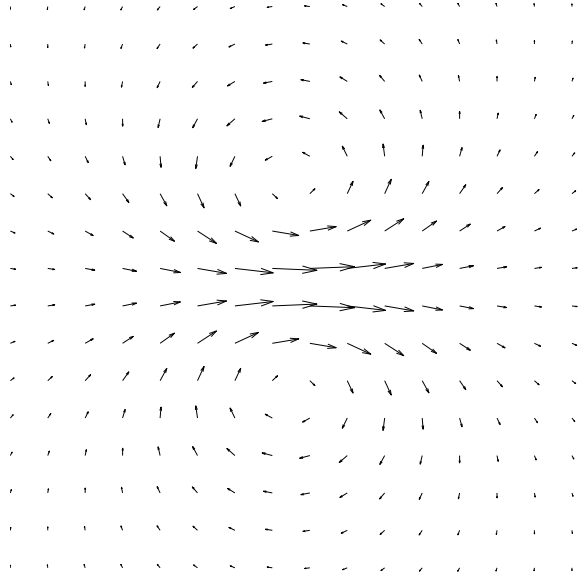,width=3in}
\caption{The dipole given by the kernel $K_{0,\et}$ in dimension 2. }\label{fig:Kplot}
\end{figure}

Finally we may also consider the limiting case $\ep=0, \et>0$. In this case $v = G_\et^{(p)} 
\ast P_{\div=0}(m)$ so $v$ has divergence zero. There is no $L$ because $v$ determines $m$ only up 
to a gradient field. However EPDiff in Oseledets's form form (\ref{eq:E1}) makes perfect sense. 
Like Euler's equation it gives geodesics on the group of volume preserving diffeomorphisms. As 
always, the energy is $E=\int v\cdot m$ and this is conserved on geodesics. Even though we have no 
$L$, we can rewrite the energy using     
$ (I-\frac{\et^2}{p}\triangle)^p v = P_{\div=0}(m)$, giving us: 
$$ E=\int v \cdot m = \int v \cdot P_{\div = 0} (m) = 
\int v\cdot (I-\tfrac{\et^2}{p}\triangle)^pv.$$ 
Alternatively, we may use the above energy to define a metric on the group  of volume 
preserving diffeomorphisms which differ from the identity by a mapping in $H^\infty$ (or 
$\mathcal S$), and our equation is just the geodesic equation on $\on{SDiff}_{H^\infty}(\mathbb R^n)$ for this metric. 
Note that the Lie algebra of $\on{SDiff}_{H^\infty}(\mathbb R^n)$ is just the space of vector 
fields $v$ in $H^\infty$ with divergence zero and 
its dual is the space of 1-forms in $H^\infty$ {\it modulo} closed 1-forms, which is the same as the space of 
exact 2-forms in $H^\infty$. In the divergence free setting there is no need to consider 1-forms tensored with 
the standard  density $dx$.
On $\on{SDiff}_{H^\infty}(\mathbb R^n)$, 
the `momentum' associated to a geodesic, defined as the dual of the 
tangent to the geodesic, is the classical vorticity of the flow and, in our formulation, we have 
simply `lifted' this to the dual Lie algebra of $\on{Diff}_{H^\infty}(\mathbb R^n)$. 

The case $p=1$ has been introduced and studied by Holm and collaborators (see \cite{HMR}, equation 
(8.29)) who use the letter $\al$ for our $\et$ and call EPDiff the $\al$-Euler equation: 
\begin{align*} 
(1-\al^2\triangle)(\p_t  v) &= -( v \cdot \nabla)(1-\al^2\triangle) v - (1-\al^2\triangle) v \cdot (D v)^t + \nabla p,\\
\div ( v)&=0.\end{align*}
You can also drop incompressibility and when $n=1$ this becomes the Camassa-Holm equation \cite{CamHolm}.

The $K$ for the $\ep=0, \et>0$ metric is just the convolution $G^{(p)}_\et \ast P_{\div = 0}$. This 
$K$ can be explicitly calculated using the fact that the Green's function $H$ is harmonic. We use: 
\begin{theorem} Let $F(x)=f(|x|)$ be any integrable $C^2$ radial function on $\R^n$. Assume $n \ge 3$. Define:
$$ H_F(x) = \int_{\R^n} \!\!\!\!\min\left(\tfrac{1}{|x|^{n-2}}, \tfrac{1}{|y|^{n-2}}\right) F(y) dy 
= \frac{1}{|x|^{n-2}}\int_{|y|\le|x|}\!\!\!\! F(y)dy + \int_{|y| \ge |x|}\!\! \frac{F(y)}{|y|^{n-2}} dy$$
Then $H_F$ is the convolution of $F$ with $\frac{1}{|x|^{n-2}}$, is in $C^4$ and:
\begin{align*}
\p_i (H_F)(x) &= -(n-2)\frac{x_i}{|x|^n}\int_{|y|\le|x|}\!\!\!\! F(y)dy \\
\p_i \p_j (H_F)(x) &= (n-2)\left(\frac{n x_i x_j - \de_{ij}|x|^2}{|x|^{n+2}} \int_{|y|\le|x|}\!\!\!\! F(y)dy 
- V_n \frac{x_i x_j}{|x|^2} F(x)\right)
\end{align*}
If $n=2$, the same holds if you replace $1/|x|^{n-2}$ by $\log(1/|x|)$ and omit the factors $(n-2)$ in the derivatives.
\end{theorem}
\begin{proof} The idea is to first check that $H_F \in C^2$ with the above expressions for its 
first and second derivatives. This is straightforward when $x \ne 0$. Near $0$, let 
$F(x) = a+b|x|^2 + o(|x|^2)$. Then one checks that:  
$$H_F(x) = \left(\int\frac{F(y)}{|y|^{n-2}} dy \right) - \frac{(n-2)aV_n}{2n}|x|^2 - \frac{(n-2)bV_n}{4(n+2)}|x|^4 + o(|x|^4)$$
hence the expressions for the first and second derivatives extend across the origin. Taking the 
trace of the matrix of second derivatives, one finds that; 
$\triangle H_F = -(n-2)V_n F.$
Since the Green's function of $-\triangle$ is $1/(n-2)V_n|x|^{n-2}$, this implies that $H_F$ is the 
convolution $F\ast(1/|x|^{n-2})$. 
\end{proof}

This applies to $F=G^{(p)}_\et$ for example, or to the limiting case where $F$ is a Gaussian, giving 
the following expression for the kernel $K_{0,\et}$ for finite $p$ or the limiting Gaussian case: 
\begin{equation} K_{0,\et}(x) = \frac{\de_{ij} |x|^2 - x_i x_j}{|x|^2} G^{(p)}_\et(x) + \frac{x_i x_j 
- \de_{ij}|x|^2/n}{|x|^2}\text{Mean}_{B_{|x|}}\left(G^{(p)}_\et\right) 
\label{eq:explicitconv} \end{equation}
where $B_a$ is the ball of radius $a$ centered at the origin.

We can summarize all possibilities in a handy table (we have changed notation slightly to use 
double subscripts $\ep, \et$ for all cases): 
$$\boxed{\begin{array}{ll} 
\text{no } L & K_{0,0}=P_{\div =0}=\de_{ij}\de_0 + (\p_i \p_j H)_{\text{distr}}\\
\text{no }L & K_{0,\et}=G^{(p)}_\et \ast P_{\div = 0} \text{ -- see above}\\
L_{\ep,0} = I-\tfrac{1}{\ep^2} \nabla \circ \div & K_{\ep,0} = \de_{ij}\de_0 + \p_i\p_j H_\ep \\
L_{\ep,\et}=\left(I-\tfrac{\et^2}{p} \triangle\right)^p \circ 
\left(I-\tfrac{1}{\ep^2} \nabla \circ \div\right) &K_{\ep,\et} = \de_{ij} G^{(p)}_\et + \p_i\p_j(G^{(p)}_\et \ast H_\ep)
\end{array}}$$

\section{Existence theorems for the $L_{\ep,\et}$ metric}

It is well known that local solutions of Euler's equation itself, that is $L_{0,0}$, exist, e.g. 
see \cite{Kato,Taylor}.  Moreover global solutions of the EPDiff equations $L_{\ep,\et}, \ep, \et > 
0, p \ge (n+3)/2$ have been shown to exist by Trouv\'e and Younes (unpublished but apparently 
implicit in the results of \cite{TrouveYounes} for geodesics in what they call `metamorphosis'). 
Their result extends easily to the EPDiff equations $L_{0,\et}$ because the kernel $K_{0,\et}$ is 
still $C^1$, which holds so long as $p \ge (n+3)/2$. The method here is based on the Lagrangian 
form (\ref{eq:lagrange}) of EPDiff. For completeness, we include the proof:

\begin{theorem}\label{th:exist0} 
Let $\ep\ge0, \et > 0, p \ge (n+3)/2$ and $K = K_{\ep,\et}$ be the corresponding kernel. For any 
vector-valued distribution $m_0$ whose components are finite signed measures, consider the 
Lagrangian equation for a time varying $C^1$-diffeomorphism $\ph(\cdot, t)$ with $\ph(x,0)\equiv x$:  
$$ \p_t \ph(x,t) = \int K(\ph(x,t)-\ph(y,t)) (D\ph(y,t))^{-1,\top} m_0(y)dy.$$
Here $D\ph$ is the spatial derivative of $\ph$. This equation has a unique solution for all time $t$.
\end{theorem}

\begin{proof} The Eulerian velocity at $\ph$ is: 
$$V_\ph(x) = \int K(x-\ph(y)) (D\ph(y))^{-1,\top} m_0(y)dy$$
and $W_\ph(x) = V_\ph(\ph(x))$ is the velocity in `material' coordinates. 
Note that because of our assumption on $m_0$, if $\ph$ is a $C^1$-diffeomorphism, 
then $V_\ph$ and $W_\ph$ are $C^1$ vector fields on $\R^n$; in fact, they are as differentiable as 
$K$ is, for suitably decaying $m$. 
The equation can be viewed as a the flow equation for the vector field $\ph \mapsto W_\ph$ 
on the union of the open sets 
$$
U_c = \big\{ \ph \in 
C^1(\R^n)^n: \|\on{Id}-\ph\|_{C^1}< 1/c,\; \det(D\ph) > c\big\},
$$
where $c>0$. The union of all $U_c$ is the group $\on{Diff}_{C^1_b}(\mathbb R^n)$ of all 
$C^1$-diffeomorphisms which, together with their inverses, differ from the 
identity by a function in $C^1(\mathbb R^n)^n$ with bounded $C^1$-norm.
We claim this vector field is locally Lipschitz on each $U_c$:
$$ \| W_{\ph_1} - W_{\ph_2}\|_{C^1} \le C.\|\ph_1 - \ph_2\|_{C^1}$$
where $C$ depends only on $c$. 
This is easy to verify using the fact that $K$ is uniformly continuous and using 
$\|D\ph\i\|\le \|D\ph\|^{n-1}/|\det(D\ph)|$. 
As a result we can integrate the vector field for short times in $\on{Diff}_{C^1_b}(\mathbb R^n)$. 
But since $(D\ph(y,t))^{-1,\top} m_0(y)$ is then again a signed finite $\mathbb R^n$-valued measure,  
$$ \int V_{\ph(\cdot,t)}(x) (D\ph(y,t))^{-1,\top} m_0(y) dx = \| V_{\ph(\cdot,t)} \|_{L_{\ep,\et}}$$
is actually finite for each $t$.
%
Using the fact that in EPDiff the $L_{\ep,\et}$-energy $\|V_{\ph(\cdot,t)}\|_{L_{\ep,\et}}$ of the 
$L_{\ep,\et}$-geodesic is constant in $t$, 
we get a bound on the norm  $\| V_{\ph(\cdot,t)}\|_{H^p}$, depending of course on $\et$ but 
independent of $t$, hence a bound on $\| V_{\ph(\cdot,t)}\|_{C^1}$. Thus $\|\ph(\cdot,t)\|_{C^0}$ 
grows at most linearly in $t$. But $\p_t D\ph = DW_\ph = DV_\ph \cdot D\ph$ 
which shows us that $D\ph$ grows at most exponentially in $t$. 
Hence $\det D\ph$ can shrink at worst exponentially towards zero, because 
$\p_t \det (D\ph) = \on{Tr}(\on{Adj}(D\ph).\p_t D\ph)$.
Thus for all finite $t$, the solution $\ph(\cdot,t)$ 
stays in a bounded subset of our Banach space and the ODE can continue to be solved.
\end{proof}

For $L_{\ep,0}$ with $\ep > 0$ we proved in a previous paper \cite{MM05} 
that the $L_{\ep,0}$-metric defined a well behaved Riemannian metric 
on the group of diffeomorphisms of $\R^n$ in that the infimum of path lengths 
joining two distinct diffeomorphisms was positive. Here we prove that for all $\ep$ and $\et$, 
including $\ep=0$ and/or $\et=0$, geodesics exist locally -- though as in the Euler case, as far as 
we know, they might become singular in finite time hence not be indefinitely prolongable --  and 
that {\it these local solutions behave continuously in the parameters $\ep, \et$}. In particular, 
as $\ep, \et \rightarrow 0$ they approach solutions of Euler's equation.     

Everything depends on proving a Sobolev estimate for the time derivative of certain energies. We 
need the following straightforward lemma: 

\begin{lemma*} If $\et\ge0$ and $\ep>0$ are bounded above, then the norm 
$$\|v\|_{k,\ep,\et}^2 = \sum_{|\al| \le k}\int \langle D^\al  L_{\ep,\et}v, D^\al v\rangle dx$$
is bounded above and below by the metric, with constants  independent of $\ep$ and $\et$:
$$
\|v\|_{H^k}^2 + \tfrac{1}{\ep^2}\|\mathrm{div}(v)\|_{H^k}^2 
+ \sum_{k+1 \le |\al|\le k+p} \et^{2(|\al|-k)}\int |D^\al v |^2 + \tfrac{1}{\ep^2}|D^\al\mathrm{div}(v)|^2
$$
\end{lemma*}

Here $H^k$ is the $k^{th}$ order Sobolev norm for the standard metric, and $D^\al$ is the partial 
derivative for the multiindex $\al$. We also often omit $dx$ at the end of integrals, and 
corresponding brackets. The proof of the lemma is obvious by expanding $L_{\ep,\et}$.

Assuming $k$ is sufficiently large, for instance $k \ge (n+2p+4)$ works, we now prove the main estimate:
$$ \boxed{ |\p_t \left(\| v\|^2_{k,\ep,\et}\right)| \le C .\| v\|_{k,\ep,\et}^3}$$
where, so long $\ep$ and $\et$ are bounded above, the constant $C$ is independent of $\ep$ and $\et$. 

Write $M_\et = (I-\tfrac{\et^2}{p} \triangle)^p$, 
so that $m_i = M_\et v_i - \tfrac1{\ep^2} M_\et\div(v)_{,i}$. 
Using EPDiff and integration by parts, the time derivative is given by:
\begin{align*} 
\tfrac12 \p_t \left(\| v\|^2_{k,\ep,\et}\right) &= \sum_{|\al|\le k} \int_{\R^n} (\p_t D^\al m, D^\al v) \\
 &=\sum_{i,j,|\al| \le k} \int_{\R^n} (-1)^{|\al|+1}
 D^{2\al} v_i.\left( v_j.m_{i,j} + m_i.v_{j,j} + m_j .v_{j,i}\right)
\end{align*}
Next replace the $m_i$ by $M_\et v_i-\tfrac{1}{\ep^2} M_\et \div(v)_{,i}$. 
Integrating the third term by parts to move the $i^{th}$ derivative of $v_j$ to the other factors 
and noting that the two terms involving the second derivative of $\div\, v$ cancel, 
one checks that the estimate can be reduced to 6 terms all of the form 
$\int D^{2\al}f \cdot g \cdot M_\et h$ with one of the triples:
\begin{align*} (f,g,h) = &(v_i,v_j,v_{i,j}),\hspace*{-.3in} &&(v_i, v_j, v_{j,i}), \hspace*{-.3in} &&(v_i, \div\, v,v_i),\\ 
&(\tfrac{\div\, v}{\ep}, v_j,\tfrac{\div\, v_{,j}}{\ep}),\hspace*{-.3in} &&(v_i, \tfrac{\div\, v}{\ep}, \tfrac{\div\, v_{,i}}{\ep}),\hspace*{-.3in} &&(\div\, v,v_j, v_j)  .\end{align*}
Next we expand $M_\et$ to $\sum_{|\be|\le p} (-1)^{|\be|}\frac{\et^{2|\be|}}{p^{|\be|}} D^{2\be}$ 
(omitting binomial constants)
and integrate by parts some more, moving a $D^\al$ from the first to second or third terms and a 
$D^\be$ from the third to first or second terms, getting terms  
$$
\frac{\et^{2|\be|}}{p^{|\be|}}\int D^{\al+\be_1}f\cdot D^{\al_2+\be_2}g \cdot D^{\be+\al_1}h
$$
where $\al = \al_1+\al_2, \be = \be_1+\be_2$. 
Now either $\al_1$ or $\al_2$ is less than or equal to $k/2$ so that the corresponding (second or  
third) term in the integrand has order at most $k/2+p+1$, hence $\le k-n/2-1$. 
Thus by the Sobolev inequalities, its sup norm is bounded by its $k^{th}$ Sobolev norm and we have:
\begin{multline*}
\frac{\et^{2|\be|}}{p^{|\be|}} \Big |  \int D^{\al+\be_1}f \cdot D^{\al_2+\be_2}g \cdot D^{\be+\al_1} h\Big | \le 
\\
\le \frac1{p^{|\be|}}
\left\{ \begin{array}{ll}
\|\et^{|\be|}D^{\al+\be_1}f\|_{L^2}\cdot \|g\|_{H^k} \cdot \|\et^{|\be|} D^{\be+\al_1}h\|_{L^2}, \qquad &\text{if } |\al_2| \le k/2, \\
\|\et^{|\be|}D^{\al+\be_1}f\|_{L^2} \cdot \|\et^{|\be|} D^{\al_2+\be_2}g\|_{L^2} \cdot \|h\|_{H^k},
\qquad &\text{if } |\al_1| \le k/2. \end{array} \right.
\end{multline*}
The first term is always bounded by $\|f\|_{k,\ep,\et}$ and so is the other $D$-term {\it except} 
in the first case with $\al_2 = 0, |\al|=k, |\be|=p $ and $h$ is a first derivative 
$\tilde h_{,\ell}$ with $\tilde h$ either a component of $v$ or $\div\, v/\ep$. In this last case, 
the third term has $k+p+1$ derivatives, so the lemma does not apply. But we can still integrate by 
parts, putting the $\ell^{\text{th}}$ derivative on the other terms. If $|\be_2| > 0$ or if 
$f=v_\ell$, this reduces again to terms bounded by the $(k,\ep,\et)$ norm. The only remaining case 
is when $f=\tilde h$, and then we have:     
$$ \int D^{\al+\be}f \cdot g \cdot D^{\al+\be}f_{,\ell} = \tfrac12 \int (D^{\al+\be}f)^2_{,\ell} \cdot g 
= -\tfrac12 \int (D^{\al+\be}f)^2 \cdot g_{,\ell}$$ 
and this finishes the proof of the estimate.

Using this estimate, we can prove:

\begin{theorem}\label{th:exist} 
Fix $k,p,n$ with $p > n/2+1, k\ge n+2p+4$ and assume 
$(\ep,\et) \in [0,M]^2$ for some $M>0$. Then there are constants $t_0, C$ 
such that for all initial conditions $v_0 \in H^{k+p+1}$, there is a unique solution 
$v_{\ep,\et}( x,t)$ of the above case of EPDiff 
(including the limiting case of Euler's equation) for $t\in [0,t_0]$. 
The solution $v_{\ep,\et}(\cdot,t)\in H^{k+p+1}$ depends continuously on $\ep,\et \in [0,M]^2$ 
and satisfies $\|v_{\ep,\et}(\cdot,t)\|_{k,\ep,\et} < C$ for all $t \in [0,t_0]$. 
\end{theorem}

For $\ep, \et >0$, existence and uniqueness for all time has been proven in \cite{TrouveYounes}. 
Their proof has been extended to the case $\ep=0, \et>0$ in Theorem \ref{th:exist0}. For 
$\ep=\et=0$, this is the well known result for Euler's equation. What remains is the new case 
$\ep>0, \et=0$. We follow a standard approach, used, for example, in \cite{Taylor}, Ch. 16 and 17. 
First consider existence. But by our estimate and Gronwall's lemma, we have a local upper bound 
{\it uniformly in} $\ep,\et$ for these solutions:     
$$ \|v_{\ep,\et}\|_{k,\ep,\et} \le C(t), t \in [0,t_0].$$
But, for $k,p$ as above, by the lemma we have $\|v\|_{H^k}\le C_1\|v\|_{k,\ep,\et}$ with 
$C_1$ independent of $\ep,\et$.
Thus the Hilbert space with the norm $\|\cdot\|_{k,\ep,\et}$  is compactly embedded in $C^1(\R^n)$ 
in the local sense that any bounded sequence for the former has a subsequence which, 
for every compact subset $K \subset \R^n$, converges in $C^1(K)$.  
Therefore $v_{\ep,\et}(t)$ lie in a `locally' compact part of the Banach space of $C^1$ functions of $(x,t)$. 
Therefore, as $\ep$ or $\et$ tend to zero, they have a convergent subsequence whose limit $v$ 
must be a solution of the corresponding EPDiff, because each equation can be written in terms of 
the corresponding kernel, and the kernels depend nicely on $\ep$ and $\et$.  
So by Gronwall's lemma again the original estimates gives $H^k$ bounds on this solution. 

Next we prove that the cluster point for $\ep\to0$ or $\et\to 0$ of the solutions $v_{\ep,\et}$ is unique. 
Let us temporarily abbreviate $L_{\ep,0}$ 
by $L$ and let $v$ and $\widetilde v$ be two solutions of EPDiff for this $L$. 
We write $u=v-\widetilde v$ for their difference and follow the ideas of the preceding estimate 
to estimate $\tfrac{d}{dt}\|u\|^2_L$.
\begin{align*}
\tfrac{d}{dt}\|u\|^2_L &= 2\int \langle \tfrac{d}{dt}Lu , u\rangle dx \qquad \text{so using summation of indices:}\\
&= -2\int u_i\cdot\left( v_j Lv_{i,j} + Lv_i\, \div v + Lv_j v_{j,i} 
- \widetilde v_j L\widetilde v_{i,j} - L\widetilde v_i \,\div \widetilde v 
- L\widetilde v_j \widetilde v_{j,i} \right)\\
&= -2\int u_i \left(u_j Lv_{i,j}+ \widetilde v_j Lu_{i,j} + Lu_i\, \div v 
+ L\widetilde v_i\, \div u + Lv_j u_{j,i} +L u_j \widetilde v_{j,i} \right)
\end{align*}
Next replace all expressions of the form $Lu_k$ by $u_k-\tfrac{1}{\ep^2}\p_k \div u$. 
Then integrate by parts by the ``div" part of the last term, that is replace 
$-u_i \cdot\tfrac{1}{\ep^2} (\div u)_{,j} \widetilde v_{j,i}$ by
$$
u_i \cdot\tfrac{1}{\ep^2} (\div u)_{,ij} \widetilde v_j + \div u \cdot \tfrac{1}{\ep^2}(\div u)_{,j} \widetilde v_j
$$
The term with the second derivative of $\div v$ cancels the term with the second derivative of 
$\div v$ arising from the second term $\widetilde v_j Lu_{i,j}$ in the above expression. With this 
and further integration by parts, we get:  
\begin{align*}
\tfrac{d}{dt}\|u\|^2_L &= -2\int u_i u_j (Lv_{i,j}+ \widetilde v_{j,i}) + u_i^2 \div v + u_i u_{i,j} \widetilde v_j + u_i u_{j,i} Lv_j + \\
&\qquad\qquad  +u_i \,\div u\, L\widetilde v_i +\tfrac{1}{\ep^2} (\div u)_{,i} (\div u \, \widetilde v_i - u_i \div v) \\
&= -2\int u_i u_j (Lv_{i,j}-Lv_{j,i}+ \widetilde v_{j,i}) + u_i^2 (\div v - \tfrac12 \div \widetilde v) + \\
&\qquad\qquad  +u_i \,\div u\,( L\widetilde v_i- Lv_i+\tfrac{1}{\ep^2}\div(v)_{,i}) 
+\tfrac{1}{\ep^2} (\div u)^2  (\div v - \tfrac12\div\widetilde v)\\
&\le C.\|u\|^2_L 
\end{align*}
where the constant depends on the strong sup bounds we have for $v$ and $\widetilde v$. By Gronwall 
again, this means that we have a growth estimate on $\|u\|^2_L$ as a function of $t$. In 
particular, if $u$ is zero at time 0, it is always zero and this proves uniqueness.  

Finally, as $\ep$ goes to zero, we again have the solutions lying in a `locally' compact part of 
$C^1$ (as above) so if there is only possible limit, they must converge to this limit and are 
continuous in $\et$. Likewise, as $\ep$ converges to zero, this solution must converge to that of 
Euler's equation.   

\section{Conserved quantities: linear and angular momentum}

We would like to derive the conservation laws from Noether's theorem using the fact that our 
geodesic equation is invariant with respect to the Euclidean group $SO(n)\ltimes \mathbb R^n$, as 
we did in our earlier paper \cite{MM07}. However, if we take 
$(X,w)\in \mathfrak{s}\mathfrak{o}(n)\ltimes \mathbb R^n$ to be the infinitesimal generator for the 
1-parameter group $(\exp(tX), tw)$, composition maps a diffeomorphism 
$\ph\in \on{Diff}_{\mathcal S}(\mathbb R^n)$ to the diffeomorphsm $\exp(tX)\o \ph + tw$. 
Unfortunately, the latter diffeomorphism no longer rapidly falls towards $\on{Id}_{\mathbb R^n}$ so 
it is not in $\on{Diff}_{\mathcal S}(\mathbb R^n)$. The infinitesimal generator for this action is        
$$
\ze_{(X,w)}(\ph) = \p_t|_0 (\exp(tX)\o \ph + tw) = X\o\ph +w.
$$
Consider a right invariant Riemannian metric $G$ on $\on{Diff}_{\mathcal S}(\mathbb R^n)$ as described 
for example in \cite{M127}, so that $G_\ph$ is an inner product on the tangent space at $\ph$, 
which is invariant under the motion group. 
Then for any geodesic $t\mapsto \ph(\cdot, t)$ the right invariant inner product 
$G_\ph(\ze_{(X,w)}(\ph), \ph_t)$ should constant in $t$, 
according to Noether's theorem in the form of  \cite[section 2.6]{M118}, 
if the action above was a left action of the motion group on $\on{Diff}_{\mathcal S}(\mathbb R^n)$. 
We could dedure this directly by taking 
$\on{Diff}_{\mathcal S}(\mathbb R^n)$ as the normal subgroup of an extension of the motion group 
which can be described as a group of diffeomorphisms which fall rapidly to ``Euclidean motions near 
infinity" and extend the metric to a right invariant one.
Instead of doing this in detail we directly check 
that the the above well defined expression is indeed constant in $t$ along each geodesic. 
Note first that
\begin{align*}
G_\ph(\ze_{(X,w)}(\ph),\ph_t) &= G_\ph(X\o \ph, \ph_t) + G_\ph(w,\ph_t) \\
&= G_{\on{Id}}(X,\ph_t\o \ph\i) + G_{\on{Id}}(w,\ph_t\o\ph\i) \\
& = G_{\on{Id}}(X,v) + G_{\on{Id}}(w,v)\quad\text{  where } 
\ph_t(\cdot,t) = v(\ph(\cdot,t),t) \\
&= \int \langle X.x, L(v)(x) \rangle\, dx + \int \langle w, L(v)(x) \rangle\, dx \\
&= \int (X.x, \widetilde m(x)) + \int (w, \widetilde m(x));
\end{align*}
the first expression viewed as a linear functional in $X\in\mathfrak{s}\mathfrak{o}(n)$ is the 
$\mathfrak{s}\mathfrak{o}(n)^*$-valued \emph{angular momentum mapping}. If we identify 
$\mathfrak{s}\mathfrak{o}(n)^*$ with $\mathfrak{s}\mathfrak{o}(n)$ via the Killing form we can 
write the angular momentum succinctly as  $\int x\wedge \widetilde m(x)$.    
Similarly the second expression leads to the \emph{linear momentum} given by $\int \widetilde m(x)$.

Let us finally prove that these momenta are conserved by the geodesic flow. We shall use the 
geodesic equation in the form $\p_t \widetilde m = -\L_v \widetilde m$. Then we have 
\begin{align*}
\p_t \int (X.x,\widetilde m(x)) &= \int (X.x,\p_t\widetilde m(x)) = -\int (X, \L_v\widetilde m)\\
&= \int (\L_v X,\widetilde m) = \int([v,X],\widetilde m) = \int (-\L_Xv,\widetilde m) \\
&=-\int \langle \L_Xv,Lv \rangle dx\qquad\text{ now use }\L_X(L) = 0\text{  and }\L_X(dx)= 0,\\
&=-\int\tfrac12 \L_X(\langle v,Lv \rangle dx) =0.
\end{align*}
For the linear momentum the proof is similar.

\section{Explicit bounds on the approximation I}

Assume you start with the same initial condition $ v(x,0)$ and integrate with both Euler's equation 
and EPDiff with $L_{\ep,\et}$. Exactly how close are they? If you look at the kernels 
$K_{\ep,\et}$, you see that the effect of $\ep>0$ is to shrink the tails of $K$ from polynomial to 
exponential and, correspondingly, to eliminate the pole of its Fourier transform at zero. On the 
other hand, the effect of $\et>0$ is to smooth the singularity of $K$ at zero or to suppress the 
high frequencies in its Fourier transform. These being opposite operations, we need to estimate 
their effects separately.       
In this section, we consider the case $\et=0$ and compare Euler's equation with that given by 
$L_{\ep,0}$. Let $ v_0(x,t)$ be the solution of Euler's equation and let $ v_\ep(x,t)$ be the 
solution of EPDiff with $L_{\ep,0}$ (below abbreviated to $L_\ep$). Our goal is to prove the 
theorem:   
\begin{theorem}\label{th:approx}
Take any $k$ and $M$ and any smooth initial velocity $v(\cdot,0)$. 
Then there are constants $t_0, C$ such that Euler's equation and 
$(\ep,0)$-EPDiff have solutions $v_0$ and $v_\ep$ respectively for $t \in [0,t_0]$ and all $\ep < M$ and these satisfy:
$$\|v_0(\cdot,t)-v_\ep(\cdot,t)\|_{H^k} \le C\ep.$$
\label{th:epsbnd} \end{theorem}

Note that by Theorem \ref{th:exist} we have essentially any bound we need on both $v_0$ and 
$v_\ep$. The $t_0$ is needed only to guarantee the bounds on the solutions derived in Theorem 
\ref{th:exist} hold for a big enough $k$ to give us the needed bounds. As above the proof is based 
on an estimate of the form:   
\begin{equation}\label{eq:estimate}
\p_t\left(\|v_0-v_\ep\|^2_{H^k}\right) \le C_1 \|v_0-v_\ep\|^2_{H^k} + \ep C_2 \|v_0-v_\ep\|_{H^k}
\end{equation}
where $C_1$ and $C_2$ depend on the initial condition $v(\cdot,0)$ and $k$ but not on $\ep$ at all 
times $t$ for which Theorem \ref{th:exist} holds for needed norm bounds on $v_0$ and $v_\ep$. 

Let $u=v_0-v_\ep$ and calculate as follows, using the geodesic equation (\ref{eqn:epsEuler}) for $L_\ep$:
\begin{align}
\notag \tfrac12 \p_t &(\| u\|_{H^k}^2 ) =  \sum_{|\al| \le k} \int D^\al u\cdot(D^\al (\p_t v_0 - \p_t v_\ep)) 
\\\notag &
= \sum_{|\al| \le k}\int D^\al u\cdot D^\al \Big\{-K_0 \ast (v_0 \cdot \nabla)v_0  
\\\notag &
\quad + K_\ep \ast (v_\ep \cdot \nabla)v_\ep + v_\ep\cdot \div (v_\ep) 
+\tfrac{\ep^2}{2}(\nabla H_\ep)\ast |v_\ep(x)|^2 + \tfrac12 (\nabla H_\ep)\ast (\div v_\ep)^2  \Big\}  
\\\notag & 
\le \sum_{|\al| \le k} \|D^\al u\|_{L^2}\cdot
\Big\{\|D^\al(K_\ep-K_0)\ast(v_\ep \cdot \nabla)v_\ep \|_{L^2} +\|D^\al(v_\ep\cdot \div(v_\ep))\|_{L^2} 
\\\notag & 
\quad  + \|D^\al (\tfrac{\ep^2}{2} (\nabla H_\ep) \ast |v_\ep|^2)\|_{L^2}
+ \| D^\al (\tfrac12(\nabla H_\ep)\ast (\div v_\ep)^2)\|_{L^2} 
\\\notag & 
\quad  +\|K_0 \ast D^\al((u\cdot \nabla) v_\ep)\|_{L^2} \Big \} 
-\sum_{|\al| \le k} \int D^\al u \cdot K_0 \ast D^\al ((v_0\cdot \nabla)u)   
\\&\notag 
\le \|u\|_{H^k} \cdot \Big\{ \|(K_0-K_\ep)\ast (v_\ep \cdot \nabla) v_\ep\|_{H^k} + \|v_\ep\cdot \div(v_\ep)\|_{H^k} 
\\\notag & 
\quad + \|\tfrac{\ep^2}{2}\nabla H_\ep \ast |v_\ep|^2\|_{H^k} + \| \tfrac12 \nabla H_\ep \ast (\div v_\ep)^2\|_{H^k} 
\\& \label{eq:epmess}\raisetag{-.25in}
 \quad  +\|((v_0-v_\ep) \cdot \nabla)v_\ep\|_{H^k} \Big\} 
- \sum_{|\al| \le k} \int D^\al(v_0-K_0 \ast v_\ep) \ast D^\al ((v_0 \cdot \nabla) u)
\end{align}

Here, in the last line, we used the fact that $K_0$, being an orthogonal projection, has norm 1 and is self-adjoint. 
Likewise $K_\ep$, after Fourier transform, at frequency $\xi$, is multiplication by a diagonal 
matrix with eigenvalues 1 and $\ep^2/(\ep^2+|\xi|^2)$; hence is also a bounded self-adjoint operator with norm 1. 

For the first term, if we abbreviate $v_\ep$ to $v$, first write:
\begin{align*} \left((K_0-K_\ep)\ast (v\cdot \nabla ) v \right) &= F - (K_0-K_\ep)\ast v \cdot \div(v) \\
\text{where } F_i &= \sum_{j,k} (K_0-K_\ep)_{ij} \ast (v_j v_k)_{,k} = \sum_{j,k} \p_k(K_0-K_\ep)_{ij} \ast v_j v_k
\end{align*}

The Fourier transform of the derivative of the difference of the $K$'s is:
$$ \left(\p_k(K_0-K_\ep)_{ij}\right)\!\!\widehat{\vphantom{\big )}\hspace*{.15in}} 
= \sqrt{-1} \xi_i \xi_j \xi_k \frac{\ep^2}{|\xi|^2(\ep^2+|\xi|^2)}$$
Thus
\begin{align*}
\|F\|^2_{H^k} &= \int (1+|\xi|^2)^k \sum_i \left| \sum_{j,k} 
\tfrac{\ep^2 \xi_i \xi_j \xi_k}{|\xi|^2(\ep^2+|\xi|^2} \widehat{v_k v_j}\right|^2 \\
& \le \int (1+|\xi|^2)^k \tfrac{\ep^4}{|\xi|^2 (\ep^2 + |\xi|^2)^2}
\sum_{j,k}(\xi_j \xi_k)^2 \cdot \sum_{j,k} |\widehat{v_j v_k}|^2 |\\\
& \le \int (1+|\xi|^2)^k \tfrac{\ep^2}{4} 
\sum_{j,k} | \widehat{v_j v_k}|^2 = \tfrac{\ep^2}{4} \sum_{j,k} \|v_j v_k \|^2_{H^k}
\end{align*}
Repeating this for $v.\div(v)$ also, we find:
$$\|(K_0-K_\ep)\ast (v \cdot \nabla)v)\|_{H^k} \le \tfrac{\ep}2 
\sum_{j,k}\|v_j v_k\|_{H^k} + \|v\cdot \div(v)\|_{H^k}) < C\ep \|v\|_{k+m,\ep,0}^2$$
for some universal constant $C$ coming from the product rule for Sobolev spaces and 
$m = \lceil n/4 \rceil$. Along the way we also derived a similar bound for
the second term in the expression (\ref{eq:epmess}). 

To treat the third and fourth  terms we need the bound on the norm of convolution with $\nabla H_\ep$:
$$ |\ep\cdot\widehat{(H_\ep)_{,i}}| = \tfrac{\ep|\xi_i|}{\ep^2+|\xi|^2} \le \tfrac{|\xi_i|}{2|\xi|},$$
hence $\|\ep \nabla H_\ep \ast f\|_{H^k} \le \tfrac12 \|f\|_{H^k}$ for any function $f$ and in particular:
\begin{align*}
\|(\tfrac{\ep^2}{2}\nabla H_\ep)\ast |v_\ep(x)|^2\|_{H^k} 
&\le \tfrac{\ep}{2} \||v_\ep|^2\|_{H^k} \le C\ep \|v_\ep\|_{k+m,\ep,0}^2 \\
\|(\tfrac12 \nabla H_\ep)\ast (\div v_\ep)^2\|_{H^k} 
&\le \tfrac{\ep}{2} \|\left(\tfrac{\div v_\ep}{\ep} \right)^2\|_{H^k} \le C\ep \|v_\ep\|_{k+m,\ep,0}^2
\end{align*}

For the fifth term in expression (\ref{eq:epmess}) we use sup bounds on $k+1$ derivatives of 
$v_\ep$ and the Sobolev inequality to obtain:
$$ \|(u \cdot \nabla)v_\ep\|_{H^k} \le C \|u\|_{H^k} \cdot \|v_\ep\|_{H^\ell}, \quad \text{with } \ell 
= k+1+\lceil \tfrac{n+1}{2} \rceil
 $$

We come to the last term in (\ref{eq:epmess}). Up to constants, we write it as:
\begin{equation} \sum_{0 \le \be \le \al, |\al| \le k} \int D^\al u \cdot (D^\be v_0 \cdot \nabla)D^{\al-\be}u 
+ \int D^\al(v_\ep - K_0\ast v_\ep) \cdot D^\al(v_0\cdot \nabla) u
\label{eq:lastbit}\end{equation}
In the first term of (\ref{eq:lastbit}), the summand with $\be = 0$ vanishes because it equals 
$\int v_0\cdot \nabla (\tfrac{|D^\al u|^2}{2})$ 
and $v_0$ has zero divergence. Using a sup norm on $D^\be v_0$, the remaining summands are bounded 
by $\|v_0-v_\ep\|_{H^k}^2$ times this sup norm. This sup norm is bounded by a universal constant 
times $\|v_0\|_{H^\ell}$ with $\ell = k+\lceil \tfrac{n+1}{2} \rceil$. To bound the second term in 
(\ref{eq:lastbit}), using the expression for $K_0$ we find $v_\ep-K_0\ast v_\ep = \nabla H_0 
\ast \div(v_\ep)$. Now calculate:  
$$ \int D^\al (v_\ep - K_0 \ast v_\ep) \cdot D^\al((v_0\cdot \nabla) u) 
= -\sum_{i,j}\int D^\al(\p_i \p_j H_0 \ast \div(v_\ep)) \cdot D^\al((v_0)_i u_j) $$
But $\p_i \p_j H_0$ has Fourier transform $(\xi_i \xi_j)/|\xi|^2$, a matrix with eigenvalues 0 and 
1, so the $L^2$ norm of the first factor is bounded by $\ep \|v_\ep\|_{k,\ep,0}$. Then, as above, 
we get a bound of the form:  
$$C\ep \|v_\ep\|_{k,\ep,0} \cdot\|v_0 \|_{H^\ell} \cdot \|u\|_{H^k}$$
with $\ell = k + \lceil \tfrac{n+1}{2} \rceil$. Now using Theorem \ref{th:exist}, we see that we 
can bound all needed norms of $v_0$ and $v_\ep$ on this time interval by norms of the initial 
condition $v(\cdot,0)$. Putting everything together, we get the asserted bound \eqref{eq:estimate}. 

To complete the proof of the Theorem, rewrite \eqref{eq:estimate} in the form 
\begin{align*}
\p_t\|u\|_{H^k} &\le \frac{C_2}2\ep  + \frac{C_1}2\|u\|_{H^k} \quad\text{  or}
\\
\|u(t)\|_{H^k} &\le \frac{C_2}2\ep t  + \int_0^t\frac{C_1}2\|u(s)\|_{H^k}ds
\end{align*}
and apply Gronwall's lemma to 
obtain $\|v_0(t)-v_\ep(t)\|_{H^k}=\|u(t)\|_{H^k}\le \ep^2 C_3 e^{C_1 t /2} = O(\ep)$
as required.

In comparing Euler's equation with EPDiff for $(\ep,0)$, a key point is that $K_0 = P_{\div = 0}$ 
and $K_\ep = K_{\ep,0}$ have identical singularities at the origin, but their difference is much 
better behaved. In fact convolution with $K_0-K_\ep$ equals   
$$\nabla \circ (\text{convolution with }J_\ep) \circ \div$$
where $J_\ep$ has Fourier transform $1/|\xi|^2(\ep^2+|\xi|^2)$. Near the origin, this looks like 
$e^{-|x|}$ in $\R^3$, has a log pole in $\R^4$ and is like $1/|x|^{n-4}$ in higher spaces. 
Considering Euler's equation and EPDiff for $(\ep,0)$ in Lagrangian form (\ref{eq:lagrange}), they 
differ only by changing the convolution on the right hand side by this term. This makes it seems 
reasonable to conjecture that if solutions of $(\ep,0)$-EPDiff do not blow up, i.e. exist for all 
time, then neither do the solutions to Euler's equation. Or conversely, if Euler's solutions do 
blow up, so do solutions of this EPDiff.

\section{Explicit bounds on the approximation II}

Now we want to compare solutions of EPDiff for $\ep>0, \et=0$ with solutions for $\ep>0, \et>0$. 
The difference here is a convolution with the Gaussian $G_\et$, so solutions with $\et>0$ are 
essentially just smoothed or low-pass version of those with $\et=0$. We will prove:  
\begin{theorem}
Let $\ep>0$. Take any $k$ and $M$ and any smooth initial velocity $v(\cdot,0)$. Then there are 
constants $t_0, C$ such that $(\ep,0)$-EPDiff and $(\ep,\et)$-EPDiff have solutions $v_0$ and 
$v_\et$ respectively for $t \in [0,t_0]$ and all $\ep, \et < M$ and these satisfy:  
$$\|v_0(\cdot,t)-v_\et(\cdot,t)\|_{H^k} \le C\et^2.$$
\label{th:etbnd} \end{theorem}

A basic tool is the simple estimate:
\begin{equation} \|f-G_\et^{(p)} \ast f\|_{L^2} \le \et^2 \|\triangle f\|_{L^2}
\label{eq:gaussbnd} \end{equation}
To prove this, just take Fourier transforms and use the elementary inequality:
$$\left(1-(1+\tfrac{\et^2}{p}|\xi|^2)^{-p}\right) \le \et^2 |\xi|^2.$$
Working as in the setup of Theorem \ref{th:approx}, let $m_0$ and $m_\et$ be the momenta corresponding to $v_0$ and $v_\et$. 
Write $u=v_0-v_\et$ and calculate the time derivative of:
$$ \|u\|^2_{k,\ep,0} = \sum_{|\al| \le k}\int \langle D^\al u, D^\al L_\ep u \rangle 
= \|u\|^2_{H^k} + \|\tfrac{\div (u)}{\ep}\|^2_{H^k}.$$
We get a lot of terms:
\begin{align}
\notag \tfrac12 \p_t \|u\|^2_{k,\ep,0} 
&= \sum_{|\al|\le k}\int D^\al u \cdot D^\al \left((\p_t m_0 - G_\et^{(p)} \ast \p_t m_\et \right)\\
\label{eq:9terms} 
&= \sum_\al \int D^\al u \cdot D^\al \left\{ (v_0\cdot \nabla)m_0- (v_\et \cdot \nabla) m_\et 
+(I- G_\et^{(p)}) \ast (v_\et \cdot \nabla) (m_\et) \right. \\
\notag & \qquad\qquad\qquad + \div(v_0) m_0 -\div(v_\et) m_\et +(I- G_\et^{(p)}) \ast (\div(v_\et) m_\et) \\
\notag & \qquad\qquad\qquad \left. + m_0\cdot (Dv_0)^t 
- m_\et \cdot (Dv_\et)^t + (I- G_\et^{(p)}) \ast m_\et \cdot (Dv_\et)^t \right\}
\end{align}
By the bound (\ref{eq:gaussbnd}), the three terms with $I-G_\et^{(p)}$ are bounded by $\|u\|_{H^k}$ times 
$\et^2\|(v_\et \cdot \nabla) m_\et\|_{H^{k+2}}$ and $\et^2\|m_\et \cdot \div(v_\et)\|_{H^{k+2}}$ and 
$\et^2\|m_\et \cdot (Dv_\et)^t\|_{H^{k+2}}$. Hence if $\ell = 1+\max(k+2, \lceil \tfrac{n}2 
\rceil)$, then, by the product rule for Sobolev norms, all three terms are bounded by $C\et^2 
\cdot\|v_\et\|_{H^\ell} \cdot \|m_\et\|_{H^\ell}$ for some constant $C$ depending only on $k$ and 
$n$. Using Theorem \ref{th:exist}, this is bounded by $C' \et^2$, where $C'$ is another constant 
now depending on the initial data as well as $k$ and $n$.    

If $\widetilde u = m_0-m_\et$, we can write the remaining terms in (\ref{eq:9terms}) as:
$$ 
(u \cdot \nabla) m_\et,\quad 
(v_0\cdot \nabla)\widetilde u,\quad 
m_\et \div(u),\quad 
\widetilde u \div(v_0), \quad 
m_\et \cdot (Du)^t, \quad 
\widetilde u \cdot (Dv_0)^t
$$
Next use the calculation:
\begin{align*}
\widetilde u &= m_0 - G_\et^{(p)} \ast m_\et -(I-G_\et^{(p)})\ast m_\et \\
&= L_\ep(v_0 - v_\et) + \text{term bnded by } \et^2 \triangle(m_\et) \text{  in }H^{k+2}\\
&= u - \tfrac{1}{\ep^2} \nabla(\div(u)) + \text{term bnded by } \et^2 \triangle(m_\et) \text{  in }H^{k+2}
\end{align*}
The $\et^2 \triangle(m_\et)$ terms are bounded like the previous ones. We finish the proof by 
applying the same tricks we have seen above to the remaining terms. Letting $C$ denote suitable 
constants depending on bounds for $v_0$ and $v_\et$, the terms with $u$, not $\widetilde u$, are 
easy:   
\begin{align*} 
&\sum_\al \int D^\al u \cdot D^\al((u\cdot \nabla) m_\et) \le C \|u\|^2_{H^k}\\
&\sum_\al \int D^\al u \cdot D^\al (\div u \cdot m_\et) \le C\|u\|_{H^k}  \|\div u \|_{H^k} \le C \|u\|^2_{k,\ep,0}\\
&\sum_\al \int D^\al u \!\cdot\! D^\al(m_\et \!\cdot\! (Du)^t) 
= - \sum_\al \int D^\al \div(u) \!\cdot\! D^\al(u\!\cdot\! m_\et) + D^\al u \!\cdot\! D^\al (u \!\cdot\! (Dm_\et)^t)\\
&\qquad \qquad \qquad \qquad \qquad \qquad \le C\|u\|^2_{H^k} + C \|u\|_{H^k} \|\div u\|_{H^k} \le C \|u\|^2_{k,\ep,0}
\end{align*}
Finally, the $\widetilde u$ terms have two more pieces, one where it is replaced by $u$ and the 
other with $\tfrac{1}{\ep^2} \nabla \div (u)$. If it is replaced by $u$, everything is bounded as 
above by $C\|u\|^2_{H^k}$ but where the usual trick is needed:  
$$ \int D^\al u \cdot D^\al((v_0\cdot \nabla) u) 
=\int D^\al u \cdot (v_0\cdot \nabla) D^\al u) + \text{terms with } \nabla D^\be u, \be < \al,$$
the latter being bounded by $\|u\|^2_{H^k}$ and the former being equal to
$$\tfrac12 \int (v_0 \cdot \nabla)|D^\al u|^2 = -\tfrac12 \int \div(v_0) \cdot |D^\al u|^2.$$
The div terms have the $1/\ep^2$ factor but also a cancellation and reduce to:
\begin{multline*} \int D^\al \tfrac{\div(u)}{\ep} \cdot D^\al( (v_0 \cdot \nabla) \tfrac{\div(u)}{\ep})=\\
-\tfrac12 \int \div(v_0) |D^\al \tfrac{\div(u)}{\ep}|^2 
+ \text{terms with }\nabla D^\be \tfrac{\div(u)}{\ep}, \be < \al.\end{multline*}
and
\begin{multline*} 
-\int D^\al u_i \cdot D^\al\left(\tfrac{\div(v_0)}{\ep} \tfrac{\div(u)_{,i}}{\ep}\right) = \\
\int D^\al \div(u) \cdot D^\al\left(\tfrac{\div(v_0)}{\ep} \tfrac{\div(u)}{\ep}\right) 
+ \int D^\al u_i D^\al \left( \tfrac{\div(v_0)_{,i}}{\ep} \tfrac{\div(u)}{\ep}\right) 
\end{multline*}
which have the needed bounds. 
Thus we have the estimate
$$
\tfrac12 \p_t \|u\|^2_{k,\ep,0} \le C_1\|u\|_{k,\ep,0}^2 + \et^2 C_2\|u\|_{k,\ep,0},
$$
and we can use Gronwall's lemma as in the end of the proof of Theorem \ref{th:approx}, to finish 
the proof.

\section{Approximating Euler solutions via landmark theory}

The great advantage of having a $C^1$ kernel is that we can now consider solutions 
in which the momentum $m$ is supported in a finite set $\{P_1, \cdots, P_N\}$, so that the 
components of the momentum field are given by
$m^i(x) = \sum_a m_{ai}\de(x-P_a)$.
The support is called the set of landmark points $\{P_1, \cdots, P_N\}$ and in this case, 
EPDiff reduces to a set of Hamiltonian ODE's based on the kernel $K = K_{\ep,\et}, \ep \ge 0, \et >0$:
$$\boxed{\begin{aligned}
\text{Energy } E&= \sum_{a,b} m_{ai} K_{ij}(P_a-P_b) m_{jb}\\
\frac{dP_{ai}}{dt} &= \sum_{b,j} K_{ij}(P_a-P_b) m_{bj} \\
\frac{dm_{ai}}{dt} &= -\sum_{b,j,k} \p_{x_i} K_{jk}(P_a-P_b) m_{aj}m_{bk}\\
\end{aligned}}$$
where $a,b$ enumerate the points and $i,j,k$ the dimensions in $\R^n$. These are essentially 
Roberts' equations from \cite{Roberts}. His paper takes $n=3$ so that the landmark points are the 
center of `circular vortex rings'. He assumes they do not get too close to each other and takes 
$K(x)$ at all $x \ne 0$ to be the Euler kernel $P_{\div=0}$, our $K_{0,0}$. He sets $K(0) = 
\de_{ij} \ka$ for a constant $\ka$ which comes out of the specific model used for each finite 
(non-infinitesimal) vortex ring. What using our kernel $K_{0,\et}$ does is just smoothly 
interpolate between the kernel $P_{\div=0}$ at points $x$ far from $0$ -- but which is singular at 
$0$ -- and a $C^1$ function near $0$ with $K(0)=\de_{ij}.\ka$.       

For some other PDEs (like the KdV or Camassa-Holm equations) solutions whose momenta are sums of 
delta distribution are called solitons. In analogy to this we can call \emph{vortex-solitons} or 
\emph{vortons} the solutions with momenta supported in finite sets.  

For every landmark tangent vector $\sum_a X_a \de(x-P_a)$ there exist a divergence free vector 
field $v$ with compact support with $v(P_a)=X_a$. Thus the space of soliton-like momenta $m(x) = 
\sum_a m_a\de(x-P_a)$ is injectively embedded in the dual of the space of divergence free vector 
fields (with compact support, of in $\mathcal S$, or in $H^\infty$). This means, that landmark 
theory as explained below is already adapted to the subgroup $\on{SDiff}_{H^\infty}(\mathbb R^n)$. 

All of our kernels have the form $K_{ij}(x) = G_1(|x|)\de_{ij} + \p_{x_i} \p_{x_j} G_2(|x|)$ hence 
at every point $x$ have eigenspaces $\R x$ and $(\R x)^\perp$. For any vector $x$, let $x = 
\rh_x\cdot u_x$, where $\rh_x = |x|$ and $u_x$ is a unit vector; and let $P_{u_x}$ be the 
projection to the subspace $\R\cdot u_x$ and $P_{u_x^\perp}$ be the projection onto the 
perpendicular subspace $(\R\cdot u_x)^\perp$. Then the matrix $K_{ij}(x)$ can be written in terms 
of two scalar functions $K_1$ and $K_2$ as      
$$ K(x) = K_1(\rh_x) P_{u_x} + K_2(\rh_x) P_{u_x^\perp}, \text { if } x \ne 0$$ 
and as $\ka\cdot I$ at the origin. If $K_1 = K_2$, then $K_{ij}$ would be a multiple of the 
identity and we would have the case studied in our previous paper \cite{MMM1}. But this never 
happens for our metrics. For example, in the $K_{0,\et}$ case, using formula 
(\ref{eq:explicitconv}) and the fact that $G^{(p)}_\et$ is a monotone decreasing function of $|x|$, 
we get:     
\begin{align*} 
&K_1(x) =\tfrac23 \text{Mean}_{B_{|x|}}(G^{(p)}_\et)(x) \\
&K_2(x) = G^{(p)}_\et(x) - \tfrac13 \text{Mean}_{B_{|x|}}(G^{(p)}_\et)(x)\\
&(K_1-K_2)(x) = \text{Mean}_{B_{|x|}}(G^{(p)}_\et)(x)- G^{(p)}_\et(x)\ge 0, \\
&\ka = \tfrac23 G^{(p)}_\et(0) > K_1(x) > K_2(x) \text{ if } x \ne 0
\end{align*}

If we differentiate the formula for $K$, we get the following formula for its derivative:
\begin{multline*} D_v K(x) = K_1'(|x|) \langle v,u_x\rangle P_{u_x} + K_2'(|x|) \langle v,u_x\rangle P_{u_x^\perp}\\ 
+ \frac{K_1(|x|)-K_2(|x|)}{|x|} \langle v, u_x^{\perp} \rangle \left(u_x \otimes u_x^{\perp} 
+ u_x^{\perp} \otimes u_x\right).\end{multline*}
Using this we can rewrite the geodesic equations in a geometric form:
\begin{align*}
E&= \sum_a \ka\cdot |m_a|^2 + \sum_{a\ne b} K_1(\rh_{ab})\langle P_{u_{ab}}m_a,P_{u_{ab}}m_b\rangle 
+ K_2(\rh_{ab})\langle P_{u_{ab}^\perp}m_a,P_{u_{ab}^\perp}m_b\rangle\\
\frac{dP_{a}}{dt} &= \ka \cdot m_a + \sum_{b\ne a} K_1(\rh_{ab}) P_{u_{ab}} m_b + K_2(\rh_{ab}) P_{u_{ab}^\perp} m_b \\
\frac{dm_{a}}{dt} &= -\sum_{b\ne a}  \left(K_1'(\rh_{ab}) \langle P_{u_{ab}}m_a, P_{u_{ab}}m_b\rangle  
+ K_2'(\rh_{ab})\langle P_{u_{ab}^\perp}m_a, P_{u_{ab}^\perp}m_b\rangle\right)u_{ab} \\
&\qquad -\sum_{a\ne b} \frac{K_1(\rh_{ab})-K_2(\rh_{ab})}{\rh_{ab}} \left( \langle m_a,u_{ab}\rangle P_{u_{ab}^\perp}m_b 
+ \langle m_b,u_{ab}\rangle P_{u_{ab}^\perp} m_a \right)
\end{align*}

One of the characteristics of these landmark space EPDiff geodesics as that when two landmarks near 
each other, they can either repel or attract. If their energy is low compared to their angular 
momentum, they repel and vice versa. When they attract, they typically spiral in towards each other 
with the momentum of each landmark point growing infinitely while their sum remains bounded. They 
do not collide in finite time. Whether this characteristic reflects developing singularity behavior 
in Euler's equation is not clear because, as soon as landmarks approach closer than $\et$, 
solutions of EPDiff are no longer close to those of Euler. This attraction is clear with only two 
landmark points but, at least in the case of the Weil-Peterson metric on cosets of 
$\on{Diff}(S^1)$, following a geodesic typically produces a hierarchical clustering of many 
landmarks (unpublished work of Sergey Kushnarev and Matt Feizsli).         

We want to look at the simplest cases of one or two landmark points. One landmark point is very 
simple:  
its momentum must be constant hence so is its velocity. 
Therefore it moves uniformly in a straight line $\ell$ from $-\infty$ to $+\infty$. 
As a geodesic in $\on{Diff}(\R^n)$, it will push everything in front of it, 
compressing points ahead of it on $\ell$ while pushing out points near $\ell$ to maintain incompressibility. 
Behind the landmark, they will be sucked back towards $\ell$ to compensate for the rarification left by its passage. 
By rotational symmetry around $\ell$ and time-reversal symmetry, the motion, 
from $t=-\infty$ to $t=+\infty$ c an only be a shear in which points are dragged forward parallel to $\ell$ 
by a distance which goes to zero as you go further from $\ell$ and goes to $\infty$ as you approach $\ell$. 
\begin{figure}[h!]
\epsfig{file=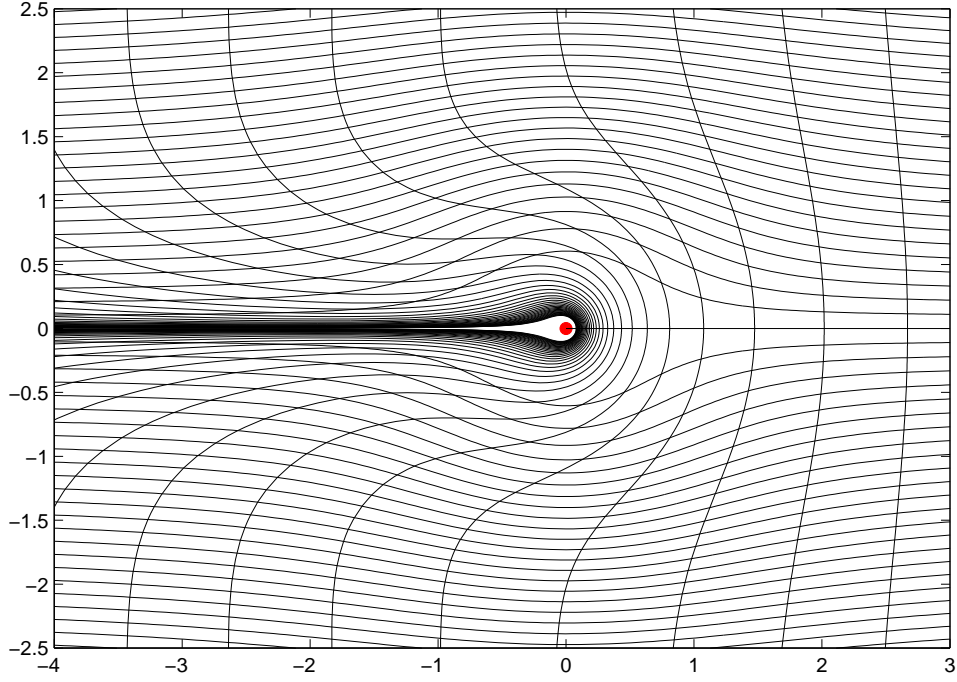,width=5in}
\caption{The result of the incompressible flow from $t=-\infty$ to $t=0$ with $n=2$, momentum 
concentrated at one point and $K=K_{0,\et}$, $\et = 1$.}\label{fig:onelandmark} 
\end{figure}
 
Now consider the case of two landmark points $P_1, P_2$ with momenta $m_1, m_2$. 
By conservation of total momentum, $m_1+m_2$ is a constant. 
We can reduce this Hamiltonian system by fixing the total momentum $\overline{m}$ and dividing by translations. 
We get a new system in the variables $\de P = P_2-P_1 $ and $\de m = m_2-m_1$ with equations of motion:
\begin{align*}
E &= \ka\cdot \left(|\de m|^2 + |\overline{m}|^2\right) 
+ \sum_{ij} K_{ij}(\de P) \cdot \left(\overline{m}_i \overline{m}_j - \de m_i \de m_j \right) \\
\frac{d(\de P_i)}{dt} &= \ka \cdot \de m_i - \sum_j K_{ij}(\de P)\cdot \de m_j \\
\frac{d(\de m_i)}{dt} &= -\sum_{jk} \p_{x_i} K_{jk}(\de P) \cdot(\overline{m}_j \overline{m}_k - \de m_j \de m_k )
\end{align*}
or, letting $\de P = \rh \cdot u$ for a unit vector $u$, in geometric form:
\begin{align*}
E &= \ka\cdot \left(|\de m|^2 + |\overline{m}|^2\right) + K_1(\rh)\left(|P_u \overline{m}|^2 - |P_u \de m|^2\right) 
\\& \qquad 
+ K_2(\rh) \left(|P_{u^\perp} \overline{m}|^2 - |P_{u^\perp} \de m|^2\right)  \\
\frac{d(\de P)}{dt} &= \ka \cdot \de m - K_1(\rh)\cdot P_u \de m - K_2(\rh)\cdot  P_{u^\perp} \de m \\
\frac{d(\de m)}{dt} &= - \left(K_1'(\rh)( |P_{u} \overline{m}|^2 - |P_u \de m|^2)  
+ K_2'(\rh)(|P_{u^\perp} \overline m|^2 - |P_{u^\perp} \de m|^2  \right)u \\
&\qquad -\frac{K_1(\rh)-K_2(\rh)}{\rh} \left( \langle \overline m,u\rangle P_{u^\perp}\overline m 
-\langle \de m, u \rangle P_{u^\perp} \de m\right)
\end{align*}

Note that the derivatives of $\de p$ and $\de m$ lie in the span of $\de P, \de m$ and 
$\overline{m}$. Thus this three dimensional space is constant in time so we can assume $\de P, 
\de m, \overline{m} \in \R^3$. The total angular momentum is:  
$$ \om = P_1 \wedge m_1 + P_2 \wedge m_2 = \tfrac{P_1+P_2}{2} \wedge \overline{m} + \tfrac12 \de P \wedge \de m.$$
If $\overline{m} = 0$, then the two vectors $\de p, \de m$ always lie in a fixed two dimensional 
space and their cross product is constant, equal to $2\om$. We can then make a further symplectic 
reduction and compute what happens in terms of the three scalar variables $\rh, 
\langle \de P,\de m\rangle, |\de m|$ which moreover must lie on one sheet of a hyperboloid: 
$$ 4|\om|^2 + \langle \de P, \de m \rangle^2 = \rh^2 \cdot |\de m |^2,\quad \rh \cdot |\de m| \ge 2|\om|.$$ 
The energy then simplifies to 
\begin{align*} 
E &= (\ka-K_2(\rh))|\de m|^2 - \tfrac{K_1(\rh)-K_2(\rh)}{\rh^2} \langle \de P, \de m \rangle^2 \\
&= \tfrac{\ka -K_1(\rh)}{\rh^2} \langle \de P, \de m \rangle^2 + 4\tfrac{\ka -K_2(\rh)}{\rh^2}|\om|^2.
\end{align*}
Its level curves on the hyperboloid must then be the geodesics. Note that as long as the kernel is 
$C^2$, $(\ka-K_1(\rh))/\rh^2$ and $(K_1(\rh)-K_2(\rh))/\rh^2$ are finite at the origin hence bounded. 

We can illustrate this in the simple case of 3-space with kernel $K_{0,1}, p=3$. As  stated above, 
then the smoothing kernel is $C^2$ and has the elementary expression  
$$G^{(3)}_1(x) = (1+|x|) e^{-|x|} = 1 - \tfrac{|x|^2}{2} + \cdots.$$ 
It's easy to calculate the mean of this function over a ball and we get:
\begin{align*} \text{Mean}_{B_{|x|}}\left(G^{(3)}_1\right)(x) 
&= 24 |x|^{-3}\left(1-e^{-|x|}\left(1+|x| + \tfrac{|x|^2}{2} + \tfrac{|x|^3}{8} \right) \right)\\
&= e^{-|x|}\left(1+\sum_{n=4}^\infty \tfrac{4}{n!}|x|^{n-3} \right)\\
&= 1 - \tfrac{3|x|^2}{10} + \cdots
\end{align*}
hence
$$ \ka = \tfrac23, \quad K_1 = \tfrac23 - \tfrac15 x^2 + \cdots, \quad K_2 = \tfrac23 - \tfrac 25 x^2 + \cdots.$$
A typical plot of the contours of $E$ in the $(\rh, |\de m|)$-plane is shown in Figure \ref{fig:Econt}. 

\begin{figure}[h!]
\epsfig{file=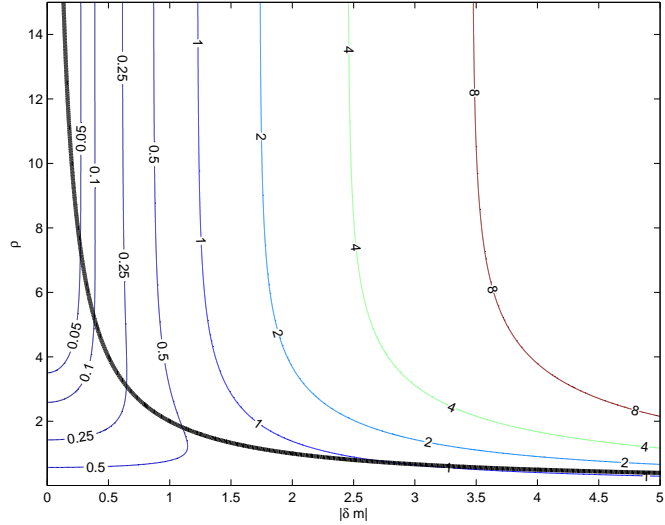,width=3.5in}
\caption{Level sets of energy for the collision of two vortons with $\overline{m} = 0$, $\et=1$, $\om = 1$. 
The coordinates are $\rh=|\de P|$ and $|\de m|$, and the state space is the double cover of the 
area above and right of the heavy black line, the two sheets being distinguished by the sign of 
$\langle \de m, \de P\rangle$. The heavy black line which is the curve $\rh\cdot |\de m| = \om$ 
where $\langle \de m, \de P\rangle=0$. Each level set is a geodesic. If they hit the black line, 
they flip to the other sheet and retrace their path. Otherwise $\rh$ goes to zero at one end of the 
geodesic.}\label{fig:Econt}     
\end{figure}

In the figure, if an orbit hits the heavy black line defined by $\rh \cdot |\de m|=\om$, then 
$\langle \de P, \de m\rangle$ is instantaneously zero and, along its orbit, changes sign. On the 
two-sheeted cover given by including this sign, this is a smooth orbit in which $\rh$ decreases to 
a minimum where $\langle \de P, \de m\rangle=0$ and then increases. One sees that there are two 
types of orbits: scattering orbits where the vortons separate infinitely at both $t=\pm \infty$ and 
$\rh$ has a minimum at some point in time; and capturing orbits which either start or end at 
infinity but spiral indefinitely, getting closer and closer, at the other limit. Which happens 
depends on the relative size of the angular momentum and the energy exactly as in the simpler case 
studied in \cite{MMM1}. Here if $E \ge (8/5)|\om|^2$, the points attract while if  $E < 
(8/5)|\om|^2$, they scatter.         

When the landmark points attract, this simple system forms higher order singularities. If we take 
coordinates so that $\de P$ is on the $x_1$-axis and $\de m$ in the $(x_1,x_2)$-plane, then for 
$\rh$ very small, we have:  
$$ P_1 = (\rh/2,0,0),\quad P_2 = (-\rh/2,0,0),\quad m_1 = \frac{1}{\rh}(C,\om,0), \quad m_2 = -\frac{1}{\rh}(C,\om,0)$$
where $2C = \langle \de P, \de m \rangle$, hence (using the limiting values of the $k$-terms in the 
formula for energy) we get $C^2 \approx \tfrac54 E - 2\om^2 >0$. Then, as these points approach 
each other, the corresponding global vector field in $\R^3$ approaches:  
\begin{align*} 
v_i(x) &= -\frac{(K_{0,1})_{i\cdot}(x+(\rh/2,0,0))-(K_{0,1})_{i\cdot}(x-(\rh/2,0,0))}{\rh} 
\cdot\left( \begin{array}{cc} C \\ \om \end{array} \right) \\
&\approx - \p_{x_1}\left((K_{0,1})_{i1}\cdot C +  (K_{0,1})_{i2}\cdot \om \right)
\end{align*}
This vector field for $\om = 2C$ is illustrated in Figure \ref{fig:swirl}. Whereas for any column 
vector $A$, $K_{01}\cdot A$ is a vortex ring with maximum norm at the origin and maximum vorticity 
along a ring centered at the origin and lying in a plane perpendicular to $A$, its derivative $v$ 
is now zero at the origin and it has maximum vorticity there. In our case, computing the 
derivatives $Dv(0)$, we find that near the origin, the flowlines of $v$ spiral in along the 
$(x_1,x_2)$-plane and shoot out along the $x_3$-axis.      
\begin{figure}[h!]
\epsfig{file=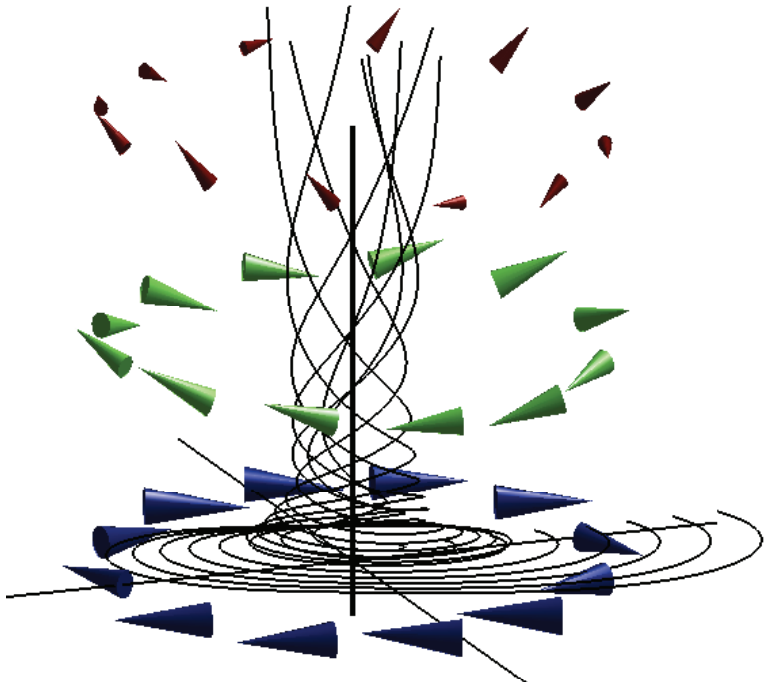,width=4in}
\caption{Streamlines and MatLab's `coneplot' to visualize the vector field given by the 
$x_1$-derivative of the kernel $K_{0,1}$ times the vector $(1,2,0)$. See text.}\label{fig:swirl} 
\end{figure}

Another case which is easy to explore is when $\overline{m}$ lies in the plane spanned by $\de P$ 
and $\de m$. The angular momentum no longer descends to a function on the $\de P, \de m$ space but 
we may numerically integrate the geodesic equations. Figure \ref{fig:mbar} shows geodesics all 
starting with the same $\de P$ and $\de m$ but with varying $\overline{m}$ fixed along the 
$y$-axis.     

\begin{figure}[h!]
\epsfig{file=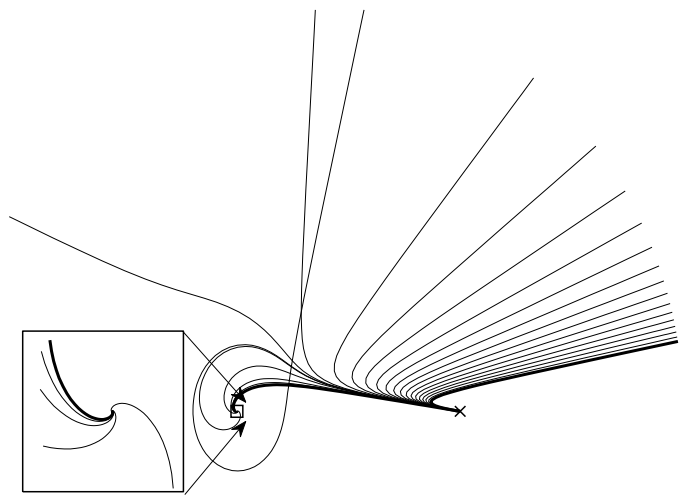,width=4in}
\caption{Geodesics in the $\de P$ plane all starting at the point marked by an {\sf X} but with 
$\overline m$ along the $y$-axis varying from 0 to 10. Here $\et=1$, the initial point is $(5,0)$ 
and the initial momentum is $(-3,.5)$. Note how the two vortons repel each other on some geodesics 
and attract on others. A blow up shows  the spiraling behavior as they collapse towards each 
other.}\label{fig:mbar}    
\end{figure}

It is extremely easy to compute landmark geodesics numerically even in much more complex situations 
and we hope that, letting $\et \rightarrow 0$, this may be a useful tool to exploring the 
instabilities of Euler's equation itself.


\begin{thebibliography}{10}

\bibitem{AbramowitzStegun}
\newblock {\em Handbook of mathematical functions with formulas, graphs, and mathematical tables.} 
\newblock Edited by Milton Abramowitz and Irene A. Stegun. 
\newblock Reprint of the 1972 edition. Dover Publications, Inc., New York, 1992

\bibitem{Arnold66} 
\newblock V.I. Arnold. 
\newblock Sur la g\'eometrie diff\'erentielle des groupes de Lie de  
dimension infinie et ses applications \`a l'hydrodynamique des  
fluides parfaits.
\newblock {\em Ann. Inst. Fourier} 16 (1966), 319--361. 

\bibitem{M118}
\newblock M.~Bauer, P.~Harms, and P.~W. Michor.
\newblock Almost local metrics on shape space of hypersurfaces in n-space,
\newblock {\em SIAM J. Imaging Sci.} 5 (2012), 244-310. 

\bibitem{Buttke}
\newblock Thomas Buttke. 
\newblock{The fast adaptive vortex method.} 
\newblock {\em J. Comput. Phys.} 93 (1991), 485.

\bibitem{CamHolm}
\newblock Roberto Camassa and Darryl Holm.
\newblock{ An Integrable shallow water equation with peaked solutions.}
\newblock {\em Physical Review Letters}, 71 (1993), 1661-1664.

\bibitem{Chorin}
\newblock Alexandre Chorin.
\newblock {\em Vorticity and Turbulence.}
\newblock Springer-Verlag, (1994).

\bibitem{Cortez}
\newblock Ricardo Cortez. 
\newblock{On the accuracy of impulse methods for fluid flow.} 
\newblock {\em SIAM J. Sci. Comput.} 19 (1998), 1290–1302

\bibitem{HMR}
\newblock Darryl Holm, Jerry Marsden and Tudor Ratiu.
\newblock{The Euler-Poincar\`e equations and Semidirect Products with Applications to Continuum Theories.}
\newblock{\em Advances in Math.}, 137 (1998), 1-81.

\bibitem{HolmMars}
\newblock Darryl Holm and Jerry Marsden.
\newblock{Momentum maps and measure-valued solutions for the EPDiff equation.}
\newblock in {\em The Breadth of Symplectic and Poisson geometry, A festschrift for Alan Weinstein},
\newblock {\em Progress in Mathematics}, 232 (2004), 203-235.

\bibitem{Hoermander83I}
\newblock Lars~H{\"o}rmander, 
\newblock\emph{The analysis of linear partial differential operators. {I}}, 
\newblock Springer-Verlag, Berlin, 1983,


\bibitem{Kato} 
\newblock Tosio Kato.
\newblock {Quasi-linear equations of evolution, with applications to partial differential equations,}
\newblock in {\it Springer Lecture Notes in Math}, 448 (1975), 27-50.

\bibitem{MMM1}
\newblock Mario Micheli, Peter Michor, David Mumford.
\newblock Sectional curvature in terms of the cometric, with applications to the Riemannian manifolds of landmarks.
\newblock {\em SIAM Journal on Imaging Sciences.} 5 (2012), 394-433.
\newblock {arXiv:1009.2637}

\bibitem{M127}
\newblock Mario Micheli, Peter W. Michor, David Mumford. 
\newblock Sobolev Metrics on Diffeomorphism Groups and the Derived Geometry of Spaces of Submanifolds. 
\newblock \emph{Izvestiya: Mathematics} 77:3 (2013), 541-570.
\newblock {\tt arXiv:1202.3677}

\bibitem{MM05}
\newblock Peter Michor, David Mumford.
\newblock Vanishing geodesic distance on spaces of submanifolds and diffeomorphisms.
\newblock {\em Documenta Mathematica}, 10 (2005), 217--245.

\bibitem{MM07}
\newblock Peter W. Michor and David Mumford.
 \newblock{An overview of the Riemannian metrics on spaces of curves using the Hamiltonian approach}. 
 \newblock{\em Applied and Computational Harmonic Analysis} 23 (2007), 74-113.
\newblock {\tt  arXiv:math.DG/0605009}

\bibitem{MM12}
\newblock Peter W. Michor and David Mumford.
\newblock{A zoo of diffeomorphism groups on $\mathbb R^n$}.
\newblock{\em Annals of Global Ananlysis and Geometry}, (2013).
\newblock {\tt  doi:10.1007/s10455-013-9380-2}. 

\bibitem{MillerEtAl93}
\newblock Michael I. Miller , Gary E. Christensen , Yali Amit and Ulf Grenander. 
\newblock Mathematical Textbook Of Deformable Neuroanatomies
\newblock {\it Proceedings National Academy of Science} 90 (1993), 11944-11948.

\bibitem{MillerEtAl02}
\newblock Michael Miller, Alain Trouv\'e and Laurent Younes.
\newblock{On the Metrics and Euler-Lagrange equations of Computational Anatomy}.
\newblock{\em Annual Review of Biomedical Engineering} (2002), 375-405.

\bibitem{Oseledets}
\newblock V.\ I.\ Oseledets.
\newblock{On a new way of writing the Navier-Stokes equations: The Hamiltonian formalism.}
\newblock {\em Communications of the Moscow Mathematical Society} (1988). 
\newblock Translation in {\em Russian Mathematics Surveys}, 44 (1989), 210-211.

\bibitem{Roberts}
\newblock P.\ H.\ Roberts.
\newblock{ A Hamiltonian theory for weakly interacting vortices.}
\newblock{\em Mathematika} 19 (1972), 169-179.

\bibitem{Taylor}
\newblock Michael E. Taylor.
\newblock{\it Partial Differential Equations III: Nonlinear Equations}.
\newblock Springer 2010.

\bibitem{TrouveYounes}
\newblock Alain Trouv\'e, Laurent Younes.  
\newblock Local geometry of deformable templates 
\newblock{\it SIAM J. Mathematical Analysis} 37 (2005), 17-59.

\bibitem{Younes10}
\newblock L.~Younes.
\newblock {\em Shapes and Diffeomorphisms}.
\newblock Springer, 2010.

\end{thebibliography}
\end{document}